\documentclass[12pt]{article}

\usepackage{times}

\usepackage[shortlabels]{enumitem}
\usepackage{enumitem}
\usepackage[utf8]{inputenc}
\usepackage{commath}
\usepackage{amsthm, amscd}
\usepackage{amsmath}
\usepackage{amssymb}
\usepackage{cite}
\usepackage{setspace}
\usepackage{ stmaryrd }
\usepackage{tikz-cd}

\usepackage{tocloft}
\usepackage{sectsty}
\usepackage{xparse}

\newcommand*\tasklabelformat[1]{#1)}

\usepackage{titlesec}

\usepackage{tasks}[newest]
\settasks{
  label = \alph* ,
  label-format = \tasklabelformat ,
  label-width  = 12pt
}

\usepackage{bigints}
\usepackage{comment}
\usepackage{mathtools}
\usepackage{mathrsfs}
\usepackage{fancyhdr}

\allowdisplaybreaks[4]

\numberwithin{equation}{section}
\usepackage{etoolbox}
\patchcmd{\thebibliography}
  {\settowidth}
  {\setlength{\itemsep}{0pt plus -10pt}\settowidth}
  {}{}
\apptocmd{\thebibliography}
  {
  }
  {}{}
\makeatletter
\newtheorem*{rep@theorem}{\rep@title}
\newcommand{\newreptheorem}[2]{%
\newenvironment{rep#1}[1]{%
 \def\rep@title{#2 \ref{##1}}%
 \begin{rep@theorem}}%
 {\end{rep@theorem}}}
\makeatother

\theoremstyle{theorem}

\newreptheorem{theorem}{Theorem}
\newtheorem{thm}{Theorem}[section]
\newtheorem*{thm*}{Theorem}
\theoremstyle{definition}
\newtheorem{prop}[thm]{Proposition}
\newtheorem*{prop*}{Proposition}
\newtheorem{defn}[thm]{Definition}
\newtheorem{lem}[thm]{Lemma}
\newtheorem{cor}[thm]{Corollary}
\newtheorem*{cor*}{Corollary}
\theoremstyle{remark}

\newtheorem{rem}[thm]{Remark}

\title{\vspace*{-1.5cm} On the Localization of the Bergman Kernel
\\
and applications to Toeplitz theory}

\author
{Siarhei Finski
}

\date{}

\usepackage[%
    left=1in,%
    right=1in,%
    top=1.1in,%
    bottom=0.8in,%
    paperheight=11in,%
    paperwidth=8.5in%
]{geometry}

\newcommand{\imun} {\sqrt{-1}}

\newcommand{\comp}{\mathbb{C}}
\newcommand{\real}{\mathbb{R}}

\newcommand{\nat}{\mathbb{N}}

\newcommand{\enmr}[1]{\text{End}{(#1)}}

\newcommand{\ccal}{\mathscr{C}}

\newcommand{\dbar}{ \overline{\partial} }

\newcommand{\tr}[1]{{\rm{Tr}} \big[ #1 \big]}

\renewcommand{\Im}{\operatorname{Im}}

\makeatletter
\DeclareFontFamily{OMX}{MnSymbolE}{}
\DeclareSymbolFont{MnLargeSymbols}{OMX}{MnSymbolE}{m}{n}
\SetSymbolFont{MnLargeSymbols}{bold}{OMX}{MnSymbolE}{b}{n}
\DeclareFontShape{OMX}{MnSymbolE}{m}{n}{
    <-6>  MnSymbolE5
   <6-7>  MnSymbolE6
   <7-8>  MnSymbolE7
   <8-9>  MnSymbolE8
   <9-10> MnSymbolE9
  <10-12> MnSymbolE10
  <12->   MnSymbolE12
}{}
\DeclareFontShape{OMX}{MnSymbolE}{b}{n}{
    <-6>  MnSymbolE-Bold5
   <6-7>  MnSymbolE-Bold6
   <7-8>  MnSymbolE-Bold7
   <8-9>  MnSymbolE-Bold8
   <9-10> MnSymbolE-Bold9
  <10-12> MnSymbolE-Bold10
  <12->   MnSymbolE-Bold12
}{}

\let\llangle\@undefined
\let\rrangle\@undefined
\DeclareMathDelimiter{\llangle}{\mathopen}%
                     {MnLargeSymbols}{'164}{MnLargeSymbols}{'164}
\DeclareMathDelimiter{\rrangle}{\mathclose}%
                     {MnLargeSymbols}{'171}{MnLargeSymbols}{'171}
\makeatother

\newenvironment{sciabstract}{}

\cftsetindents{section}{0em}{1.7em}

\sectionfont{\large}

\titlespacing*{\section}
{0pt}{5pt}{5pt}

\begin{document}

\maketitle 

\vspace*{-0.7cm}

\vspace*{0.3cm}

\begin{sciabstract}
  \textbf{Abstract.} 
	For a compact complex manifold endowed with a big line bundle and a Radon measure, we study the localization phenomena of the associated Bergman (or Christoffel-Darboux) kernel. 
	For Bernstein-Markov measures, this results in the determination of the limiting off-diagonal Bergman measure, thereby confirming a conjecture of Zelditch. 
	We then turn to applications in the theory of Toeplitz operators, showing in particular that they form an algebra under composition.
	Building on this, we then show that for Bernstein-Markov measures, the spectrum of Toeplitz operators equidistributes.
\end{sciabstract}

\pagestyle{fancy}
\lhead{}
\chead{On the Localization of the Bergman Kernel}
\rhead{\thepage}
\cfoot{}


\newcommand{\Addresses}{{
  \bigskip
  \footnotesize
  \noindent \textsc{Siarhei Finski, CNRS-CMLS, École Polytechnique F-91128 Palaiseau Cedex, France.}\par\nopagebreak
  \noindent  \textit{E-mail }: \texttt{finski.siarhei@gmail.com}.
}} 

\vspace*{0.25cm}

\par\noindent\rule{1.25em}{0.4pt} \textbf{Table of contents} \hrulefill

\vspace*{-1.5cm}

\tableofcontents

\vspace*{-0.2cm}

\noindent \hrulefill


\section{Introduction}\label{sect_intro}
	The goal of this paper is to study a localization phenomenon for the Bergman kernel and to apply this study to the theory of Toeplitz operators.
	\par 
	To state our results precisely, we fix a compact complex manifold $X$, $\dim_{\comp} X = n$, and a \textit{big line bundle} $L$ over $X$, i.e. such that for some $c > 0$, we have $\dim H^0(X, L^{\otimes k}) \geq c k^n$, for $k \in \nat$ big enough.  
	Recall that a subset $E \subset X$ is called \textit{pluripolar} if in each local chart we have $E \subset \{\phi = -\infty\}$ for some plurisubharmonic (\textit{psh}) function $\phi$.
	We fix a continuous Hermitian metric $h^L$ on $L$ and a Borel measure $\mu$ on $X$ which \textit{does not give full mass to pluripolar subsets}, i.e. such that for any pluripolar subset $E \subset X$, we have $\mu(E) \neq \mu(X)$.
	This condition is weaker than requiring $\mu$ to be \textit{non-pluripolar}, i.e. such that for any pluripolar subset $E \subset X$ we have $\mu(E) = 0$.
	We denote by ${\textrm{Hilb}}_k(h^L, \mu)$ the positive semi-definite form on $H^0(X, L^{\otimes k})$ defined for arbitrary $s_1, s_2 \in H^0(X, L^{\otimes k})$ as follows
	\begin{equation}\label{eq_defn_l2}
			\langle s_1, s_2 \rangle_{{\textrm{Hilb}}_k(h^L, \mu)} = \int \langle s_1(x), s_2(x) \rangle_{(h^L)^{k}} \cdot d \mu(x).
	\end{equation}
	Remark that since $\mu$ does not give full mass to analytic subsets (as every analytic subset is pluripolar), the above form (\ref{eq_defn_l2}) is positive definite.
	\par 
	We assume that $k$ is large enough so that $n_k > 0$, where $n_k := \dim H^0(X, L^{\otimes k})$.
	The central object of this article is the Bergman (or Christoffel-Darboux) kernel, \cite{Bergman}, defined as
	\begin{equation}\label{eq_bergm_kern}
		B_k(x, y) := \sum_{i = 1}^{n_k} s_i(x) \cdot s_i(y)^*  \in L^k_x \otimes (L^k_y)^*,
	\end{equation}
	where $s_i$, $i = 1, \ldots, n_k$, is an orthonormal basis of $(H^0(X, L^{\otimes k}), {\textrm{Hilb}}_k(h^L, \mu))$ (an easy verification shows that (\ref{eq_bergm_kern}) doesn't depend on the choice of the basis).
	\par 
	To grasp the intuitive meaning behind the Bergman kernel, we denote by $\Sigma_k$ the base loci of $L^k$, i.e. $\Sigma_k := \{ x \in X  :  s(x) = 0 \text{ for any } s \in H^0(X, L^{\otimes k}) \}$.
	For $x \in X \setminus \Sigma_k$, we denote by $s_{k, x} \in H^0(X, L^{\otimes k})$ the \textit{peak section} at $x$ with respect to the scalar product ${\textrm{Hilb}}_k(h^L, \mu)$; recall that this means that $s_{k, x}$ is the unique section (up to a unimodular complex multiplication) which maximizes the ratio $|s(x)|_{(h^L)^k} / \| s \|_{{\textrm{Hilb}}_k(h^L, \mu)}$ among all $s \in H^0(X, L^{\otimes k}) \setminus \{0\}$.
	\par 
	Directly from the definition of the Bergman kernel and the fact that it doesn't depend on the choice of an orthonormal basis, we deduce 
	\begin{equation}\label{eq_bk_id}
			B_k(x, y) = 
			\begin{cases} s_{k, x}(x) \cdot s_{k, x}(y)^*, \quad & x \in X \setminus \Sigma_k, y \in X,
			\\
			0, &\text{otherwise.}
			\end{cases}
	\end{equation}
	\par 
	The description of peak sections suggests that, as $k \to \infty$, the values $y \mapsto |s_{k, x}(y)|_{(h^L)^k}$ become increasingly concentrated near $y = x$ for any fixed $x \in X$.
	\par 
	For certain “regular” measures, the above heuristic holds in a strong form; see Remark \ref{rem_off_diag} for a brief survey.
	In general, however, naive versions of the phenomenon break down: one does not have pointwise convergence to zero, see Section \ref{sect_off_diag}, nor even weak convergence of the norm at individual points in the support of $\mu$, see Breuer-Last-Simon \cite[\S3-5]{NevaiCond}.
	The main result of this article shows, nevertheless, that the heuristic can be recovered in an \emph{averaged} form.
	\par 
	Before turning to the general formulation of our results, we illustrate a connection with orthogonal polynomials.
	Let $\mu$ be a Borel measure on $\comp$ so that $K := {\textrm{supp}}(\mu)$ is non-finite and compact. 
	Consider a sequence $p_i(z)$, $i \in \nat$, $z \in \comp$, of orthogonal polynomials associated with $\mu$, cf. \cite{SzegoBookOrt}. 
	Recall that it means that each $p_i$ is a polynomial of degree $i$ in $z$, normalized to have unit norm in $L^2(\mu)$, and orthogonal to all $p_j$ with $j < i$.
	The non-finiteness of the support of $\mu$ assures that the scalar product (\ref{eq_defn_l2}) is positive definite and hence $p_j$ are well-defined (up to a multiplication by a unimodular constant).
	The associated Christoffel-Darboux kernel, \cite{Christoffel}, \cite{Darboux}, is defined by
	\begin{equation}\label{eq_cd_orth}
		B_k^{{\rm{ort}}}(x, y) := \sum_{i = 0}^{k} p_i(x) \cdot \overline{p_i(y)}, \qquad x, y \in \comp.
	\end{equation}
	To realize this Christoffel-Darboux kernel as a special case of the Bergman kernel, we consider the projective space $X := \mathbb{P}^1$ and the hyperplane bundle $L = \mathscr{O}(1)$. 
	We view $\mathbb{C}$ as an affine chart in $\mathbb{P}^1 = \mathbb{C} \cup \{\infty\}$ with the standard holomorphic coordinate $z$, and consider a non-zero section $\sigma \in H^0(X, L)$ vanishing at $\infty$. 
	The division by $\sigma^{k}$ gives an isomorphism between $H^0(X, L^{\otimes k})$ and the space of polynomials of degree $\leq k$ in $z$. 
	If we now endow $L$ with a continuous metric $h^L$, such that $|\sigma(x)|_{h^L} = 1$ for all $x \in K$, then the $L^2$-norm (\ref{eq_defn_l2}) coincides, under the above identification, with the $L^2(\mu)$-norm on the space of polynomials. 
	In particular, (\ref{eq_bergm_kern}) then specializes to (\ref{eq_cd_orth}).
	\par 
	We now come back to our general setting of a complex manifold and a big line bundle.
	We introduce the following sequence of measures on $X \times X$:
	\begin{equation}\label{eq_mu_begmm}
		\mu_k^{\rm{Berg}} := \frac{1}{n_k} |B_k(x, y)|^{2}_{(h^L)^k} \cdot d\mu(x) d\mu(y), \quad k \in \nat.
	\end{equation}
	It follows immediately from (\ref{eq_bergm_kern}) that each $\mu_k^{\rm{Berg}}$ is a \textit{probability measure}.
	The following result confirms the heuristics explained after (\ref{eq_bk_id}), and it is the main result of this paper.
	\par 
	\begin{thm}\label{thm_off_diag}
		For any Borel measure $\mu$ which does not give full mass to pluripolar subsets, the corresponding measures $\mu_k^{\rm{Berg}}$ do not asymptotically place mass away from the diagonal.
		In other words, for any compact subset $K \subset X \times X$, not intersecting the diagonal, we have
		\begin{equation}\label{eq_thm_off_diag}
			\lim_{k \to \infty} \int_{K} \mu_k^{\rm{Berg}} = 0.
 		\end{equation}
	\end{thm}
	\begin{rem}\label{rem_off_diag}
		There is a large body of work which establishes versions of Theorem \ref{thm_off_diag} for some specific measures. A comprehensive survey would take us too far afield, so we merely cite a few representative results. 
		In the complex geometric setting -- when the measure arises from a volume form -- see the works of Dai-Liu-Ma \cite{DaiLiuMa}, Christ \cite{ChristBergmanOff}, Ma-Marinescu \cite{MaMarOffDiag}, Berman \cite{BermanEnvProj}.
		For measures supported on totally real algebraic submanifolds, see Berman-Ortega-Cerdá \cite{BermanOrtega}.
		In the context of orthogonal polynomials, we refer to Nevai \cite{Nevai}, Lubinsky \cite{LubinskyUnivers}, Breuer-Last-Simon \cite{NevaiCond}.
		\par 
		We note that while Theorem \ref{thm_off_diag} appears to be new even in the context of orthogonal polynomials, the proof of it simplifies considerably in this special case (only Section \ref{sect_supp} is relevant), and in fact a stronger statement holds under weaker assumptions: according to Remark \ref{rem_fuj}, for arbitrary compactly supported measures $\mu$ on $\comp$ with \textit{non-finite support} and a compact $K \subset \comp \times \comp$, not intersecting the diagonal, there is $C > 0$, so that in the notations of (\ref{eq_cd_orth}), for any $k \in \nat$, we have
		\begin{equation}\label{eq_rem_fuj}
			\iint_{K} \big|B_k^{{\rm{ort}}}(x, y)\big|^2 \cdot d \mu(x)  d \mu(y)
			\leq 
			C,
		\end{equation}
		which of course refines (\ref{eq_thm_off_diag}), as $n_k = k + 1$.
	\end{rem}
	\par 
	In the setting of the Bernstein-Markov measures, the above theorem can be made more precise.
	To explain this, let us recall some notations first.
	We say that $\mu$ is \textit{Bernstein-Markov} with respect to $h^L$, if for each $\epsilon > 0$, there is $C > 0$, such that for any $k \in \nat$, we have
	\begin{equation}\label{eq_bm_clas}
		{\textrm{Ban}}_k^{\infty}(K, h^L)
		\leq
		C
		\cdot
		\exp(\epsilon k)
		\cdot
		{\textrm{Hilb}}_k(h^L, \mu), \quad \text{where} \quad K = {\textrm{supp}} \mu,
	\end{equation}
	and ${\textrm{Ban}}_k^{\infty}(K, h^L)$ stands for the $L^{\infty}(K)$-norm on $H^0(X, L^{\otimes k})$ induced by $h^L$. 
	For a survey on the Bernstein-Markov property, we recommend \cite{BernsteinMarkovSurvey}.  
	Here, we only mention that the Lebesgue measure on a smoothly bounded domain in $X$, or on a totally real compact submanifold of real dimension $n$, is Bernstein-Markov, see \cite[Corollary 1.8]{BerBoucNys}, \cite[Proposition 3.6]{BernsteinMarkovSurvey} and \cite[Theorem 1.3]{MarinescuVu}, and also the recent survey \cite{MarinescuVuSurvey}.
	\par 
	Now, we fix a non-pluripolar subset $E \subset X$. 
	Following Siciak \cite{SiciakExtremal}, Guedj-Zeriahi \cite{GuedZeriGeomAnal}, we define the \textit{psh envelope} $h^L_{E}$ associated with $E$ as
	\begin{equation}\label{defn_env}
		h^L_E
		=
		\inf \Big\{
			h^L_0 \text{ with psh potential }: h^L_0 \geq h^L \text{ over } E
		\Big\}.
	\end{equation}
	When $L$ is ample, the non-pluripolarity of $E$ assures that the metric $h^L_E$ has a bounded potential, cf. \cite[Theorem 9.17]{GuedZeriGeomAnal}.
	When $L$ is merely big, the potential of $h^L_E$ might have singularities.
	In either case, the lower semi-continuous regularization $h^L_{E *}$ of $h^L_E$ is a metric with a psh potential, and it has the minimal singularities, cf. \cite[\S 1.2]{BermanBouckBalls} for the necessary definitions.
	\par 
	Recall that the equilibrium measure associated with a non-pluripolar $(K, h^L)$ is defined as
	\begin{equation}\label{eq_equil_meas}
		\mu_{\mathrm{eq}}(K, h^L) := \frac{1}{{\rm{vol}}(L)} c_1(L, h^L_{K *})^n.
	\end{equation}
	Above, ${\rm vol}(L)$ denotes the volume of the line bundle $L$, defined by
	\begin{equation}
		{\rm{vol}}(L) := \lim_{k \to \infty} \frac{n!}{k^n} \dim H^0(X, L^{\otimes k}),
	\end{equation}
	a limit that exists by Fujita’s theorem \cite{Fujita} and is strictly positive under our assumption that $L$ is big. 
	The positive $(1,1)$-current $c_1(L, h^L_{K *})$ is defined as $ c_1(L, h^L_{K *}) := c_1(L, h^L_0) + \frac{\imun}{2 \pi} \partial \bar{\partial} \phi_K$, where $h^L_{K *} = h^L_0 \cdot \exp(-\phi_K)$ and $h^L_0$ is an arbitrary smooth metric on $L$. 
	It is generally non-smooth by the above discussion.
	When $L$ is ample, the top wedge power $c_1(L, h^L_{K*})^n$ can be defined in the sense of Bedford-Taylor \cite{BedfordTaylor}, since $h^L_{K*}$ admits a bounded potential. 
	However, as we assume $L$ is merely big, the wedge product $c_1(L, h^L_{K*})^n$ appearing in (\ref{eq_equil_meas}) is interpreted via the non-pluripolar product developed by Boucksom-Eyssidieux-Guedj-Zeriahi \cite{BEGZ}, which generalizes the Bedford-Taylor construction. 
	In this case, $\mu_{\mathrm{eq}}(K, h^L)$ is a probability measure on $X$ supported in $K$, see \cite[Proposition 1.10]{BermanBouckBalls}.
	The following result settles a question left open by Zelditch \cite[after Theorem 2.5]{ZelditchErgodSurvey} and provides a refinement of Theorem \ref{thm_off_diag}, albeit under an additional assumption.
	\begin{thm}\label{thm_off_diag_bm}
		For any Bernstein-Markov measure $\mu$ with a non-pluripolar support, the measures $\mu_k^{\rm{Berg}}$ converge weakly, as $k \to \infty$, to the measure $\Delta_* \mu_{\mathrm{eq}}(K, h^L)$, where $\Delta: X \to X \times X$ is the diagonal embedding.
	\end{thm}
	\begin{rem}
		As we shall explain in Section \ref{sect_off_diag}, Theorem \ref{thm_off_diag_bm} refines the convergence of the diagonal Bergman measures, due to Berman-Boucksom-Witt Nyström \cite{BerBoucNys} and Bloom-Levenberg \cite{BloomLeven1}, \cite{BloomLeven2}.
		We, however, stress out that our proof depends on \cite{BerBoucNys}, and the only novel aspect of Theorem \ref{thm_off_diag_bm} is the proof of the asymptotic concentration of the mass of $\mu_k^{\mathrm{Berg}}$ along the diagonal, i.e. Theorem \ref{thm_off_diag}.
		We also underline that Theorem \ref{thm_off_diag} can be established rather easily for measures satisfying the Bernstein-Markov assumption, see Section \ref{sect_off_diag}.
	\end{rem}
	\par
	We state now a complementary form of the localization principle, phrased in terms of the \textit{diagonal} Bergman measure, which is a probability measure on $X$ defined as
	\begin{equation}\label{eq_diag_bm}
		\mu_k^{\rm{Berg}, \Delta}
		:=
		\frac{1}{n_k} B_k(x, x) \cdot d \mu(x).
	\end{equation}
	The same heuristic discussed after (\ref{eq_bk_id}) suggests that although the diagonal Bergman kernel is a global invariant of the measure, variations in the weight of the measure should affect the asymptotic value of the diagonal Bergman kernel at a given point only through the local change of the weight at that point, rather than globally. This observation is formalized in our next theorem, for which we fix two Borel measures $\mu_1$ and $\mu_2$ on $X$, so that for a certain continuous $f : X \to \real$, $\mu_2 = \exp(f) \cdot \mu_1$.
	We denote by $\mu_{k, i}^{\rm{Berg}, \Delta}$, $k \in \nat$, $i = 1, 2$, the associated diagonal Bergman measures.
	\begin{thm}\label{thm_diag_local}
		For $\mu_1$ and $\mu_2$ as above, the sequence of measures $\mu_{k, 1}^{\rm{Berg}, \Delta} - \mu_{k, 2}^{\rm{Berg}, \Delta}$ converges weakly to the zero measure, as $k \to \infty$, if $\mu_1$ or $\mu_2$ do not give full mass to pluripolar subsets.
	\end{thm}
	\begin{rem}
		It is natural to ask whether Theorem \ref{thm_diag_local} admits a proof that does not rely on Theorem \ref{thm_off_diag}. 
		If so, Theorem \ref{thm_off_diag} could in turn be derived by the methods of Section \ref{sect_off_diag}.
	\end{rem}
	\par 
	Our primary motivation for the study of the localization phenomena of the Bergman kernel lies in the theory of Toeplitz operators.
	Let us recall the necessary notations: for $f \in \ccal^0(X, \real)$, $k \in \nat^*$, we define the Toeplitz operator $T_k(f) \in {\enmr{H^0(X, L^{\otimes k})}}$ as $T_k(f) := B_k \circ M_k(f)$, where $B_k : \ccal^0(X, L^{\otimes k}) \to H^0(X, L^{\otimes k})$ is the orthogonal (Bergman) projection to $H^0(X, L^{\otimes k})$ with respect to (\ref{eq_defn_l2}), and $M_k(f) : H^0(X, L^{\otimes k}) \to \ccal^0 (X, L^{\otimes k})$ is the multiplication map by $f$.
	\par 
	By considering the products $\langle T_k(f) s_1, s_2 \rangle_{{\textrm{Hilb}}_k(h^L, \mu)}$ for arbitrary $s_1, s_2 \in H^0(X, L^{\otimes k})$, we see that $T_k(f)$ depends solely on the restriction of $f$ to $K$.
	As by Tietze-Urysohn-Brouwer extension theorem, any function from $\ccal^0(K, \real)$ admits an extension to $\ccal^0(X, \real)$, we can thus extend the definition of $T_k(f)$ for any $f \in \ccal^0(K, \real)$.
	\par 
	One of the key features of Toeplitz operators is their asymptotic algebra property: for $f, g \in \ccal^0(K, \real)$, the composition $T_k(f) \circ T_k(g)$ is asymptotically close to $T_k(f \cdot g)$, as $k \to \infty$. 
	This phenomenon appears to have been first observed by Grenander-Szeg{\H{o}} in \cite[\S 7.4]{GrenanSzego} in the context of certain orthogonal polynomials, see Section \ref{sect_mat}. 
	It was later established by Bordemann-Meinrenken-Schlichenmaier \cite{BordMeinSchli} when $\mu$ is a volume form, $L$ is ample and $h^L$ is smooth and positive; see also Ma-Marinescu \cite{MaMarToepl}, \cite{MaMarBTKah}. 
	The following result shows that, somewhat unexpectedly, this phenomenon holds in great generality.
	\begin{thm}\label{cor_alg}	
		For any Borel measure $\mu$ which does not give full mass to pluripolar subsets, the space of Toeplitz operators verifies the asymptotic algebra property.
		In other words, for any $f, g \in \ccal^0(K, \real)$, $p \in [1, +\infty[$, we have
		\begin{equation}\label{eq_toepl_comp_as}
			\lim_{k \to \infty} \| T_k(f) \circ T_k(g) - T_k(f \cdot g) \|_p = 0,
		\end{equation}
		where $\| \cdot \|_p$ is the $p$-Schatten norm, defined for an operator $A \in {\enmr{V}}$, of a finitely-dimensional Hermitian vector space $(V, H)$ as $\| A \|_p = (\frac{1}{\dim V} {\rm{Tr}}[|A|^p])^{\frac{1}{p}}$, $|A| := (A A^*)^{\frac{1}{2}}$.
	\end{thm}
	\begin{rem}
		Apart from the case when $\mu$ is a volume form, $L$ is ample and $h^L$ is smooth and positive, the author is not aware of any general condition on $\mu$ for which the $p$-Schatten norm above could be replaced by the operator norm. 
		Due to possible connections to symplectic geometry in light of the foundational works \cite{BordMeinSchli} and Ma-Marinescu \cite{MaMarBTKah}, it would also be interesting to classify measures $\mu$, which allow for a refinement of (\ref{eq_toepl_comp_as}) to include the lower-order terms in $k$, and ultimately to analyze the commutator $\frac{1}{k} [T_k(f), T_k(g)]$.
	\end{rem}
	\par 
	As an application of Theorem \ref{cor_alg}, we get the following equidistribution result for the spectral measures of Toeplitz operators.
	\begin{thm}\label{thm_distr}
		For any Bernstein-Markov measure $\mu$ with a non-pluripolar support, and any $f \in \ccal^0(K, \real)$, for and any continuous $g: \real \to \real$, we have
		\begin{equation}
			\lim_{k \to \infty} \frac{1}{n_k} \sum_{\lambda \in {\rm{Spec}} (T_k(f))} g(\lambda)
			=
			\int_X g(f(x)) d \mu_{\mathrm{eq}}(K, h^L)(x).
		\end{equation}
	\end{thm}
	\begin{rem}
		When $X = \mathbb{P}^1$, $L = \mathscr{O}(1)$, and $\mu$ is the Lebesgue measure on the unit circle $\mathbb{S}^1 \subset \mathbb{C} \subset \mathbb{P}^1$, the result recovers Szeg\"o first limit theorem \cite{SzegoFST}, see Section \ref{sect_mat} for details.
		When $\mu$ is a volume form, $L$ is ample and $h^L$ is smooth and positive, Theorem \ref{thm_distr} was established by Boutet de Monvel-Guillemin \cite{BoutGuillSpecToepl}. 
		See also Berman \cite[Theorem 2.7]{BermanSuperToepl} for a further result treating the semi-positive metric $h^L$.
	\end{rem}
	\par 
	The structure of this article is as follows.
	In Section \ref{sect_local_main}, we establish Theorem \ref{thm_off_diag_bm}.
	We also prove Theorem \ref{thm_off_diag}, subject to certain auxiliary results that will be derived later in Section \ref{sect_cd_diag}, where we also complete the proof of Theorem \ref{thm_diag_local}.
	Finally, Section \ref{sect_appl} is devoted to applications, including the proofs of Theorems \ref{cor_alg} and \ref{thm_distr}.
	\par
	\textbf{Acknowledgement.}
	The author thanks Norm Levenberg for valuable discussions and insightful comments on an earlier version of this manuscript, which led to substantial improvements in the first draft.
	He also thanks Robert Berman for drawing attention to \cite{BermanOrtega}, Jonathan Breuer for sharing some details concerning \cite{NevaiCond}, Franck Wielonsky for allowing to reproduce his example in Section \ref{sect_bergm_pp} and the anonymous referee for the careful reading. This work was supported by the CNRS, École Polytechnique, and in part by the ANR projects QCM (ANR-23-CE40-0021-01), AdAnAr (ANR-24-CE40-6184), STENTOR (ANR-24-CE40-5905-01) and CanQuantFilt (ANR-25-ERCS-0009).
	\par
	\textbf{Conflict of Interest and Data Availability Statements.}
	The author declares that there is no conflict of interest regarding the publication of this manuscript.
	This article does not involve the generation or analysis of datasets. All results are theoretical, and no datasets were generated or analyzed during the current study.
	
	\section{Localization along the diagonal of the Bergman measures}\label{sect_local_main}
	The main objective of this section is to analyze the off-diagonal behavior of the Bergman measures.
	We begin in Section \ref{sect_exampl} with several explicit computations, which serve to highlight the sharpness of Theorem \ref{thm_off_diag}.
	The remainder of the section is devoted to the proof of Theorem \ref{thm_off_diag}. 
	In Section \ref{sect_off_diag}, we prove the theorem under the additional assumption that the measure $\mu$ satisfies the Bernstein-Markov property, thereby establishing Theorem \ref{thm_off_diag_bm} as a consequence.
	The general case is shown in Sections \ref{sect_local} and \ref{sect_supp} modulo a certain statement established later in Section \ref{sect_cd_diag}.
	
	\subsection{Case studies of off-diagonal behavior}\label{sect_exampl}
	The main goal of this section is to discuss to which extent Theorem \ref{thm_off_diag} is sharp by studying explicitly  some examples of the off-diagonal behavior of the Bergman measures.
	\par 
	The off-diagonal expansion of the Bergman kernel has been thoroughly investigated by Christ \cite{ChristBergmanOff} and Ma-Marinescu \cite{MaMarOffDiag} when $\mu$ is a volume form, $L$ is ample and $h^L$ is smooth and positive.
	Their results imply that in the notations of Theorem \ref{thm_off_diag}, there exists a constant $c > 0$, depending on $K$, such that $\int_{K} \mu_k^{\mathrm{Berg}} \leq \exp(-c \sqrt{k})$.
	This estimate hinges on the spectral gap property of the Kodaira Laplacian, a property which fails even for smooth semi-positive metrics $h^L$ on non-ample $L$, see Donnelly \cite{DonnellyGap}. 
	Alternatively, the exponential decay can also be obtained through Ohsawa-Takegoshi extension theorem, see \cite[Theorem 1.3]{BermanBulkUniv}.
	In the broader setting considered here, the two approaches do not seem to be applicable.
	Moreover, the following example shows that this exponential decay doesn't hold even for very “nice" measures.
	\par 
	\textbf{Example 1.}
	We embed $\mathbb{S}^1$ in $X := \mathbb{P}^1$ as one of the great circles (for concreteness given by $\theta \mapsto [1: \exp(i \theta)] \in \mathbb{P}^1$, $\theta \in [0, 2 \pi[$, where $[1 : z] \in \mathbb{P}^1$, $z \in \comp$, is a standard affine chart) and denote by $\mu$ the Lebesgue measure on $\mathbb{S}^1$, viewed as a measure on $X$.
	Remark that as $\mathbb{S}^1$ is totally real of maximal dimension, it is non-pluripolar, see \cite{SadullaevReal}, cf. \cite[Exercise 4.39.5]{GuedjZeriahBook}. 
	We then take $L := \mathscr{O}(1)$ endowed with the Fubini-Study metric $h^{FS}$.
	By the results recalled in the Introduction, the measure $\mu$ is Bernstein-Markov for $(\mathbb{S}^1, h^{FS})$. 
	\par 
	An easy calculation shows that with respect to the $L^2$-product associated with the above $\mu$, the standard basis of monomials of $H^0(X, L^{\otimes k})$ given by $z^i$, $i = 0, \ldots, k$, in the already mentioned affine chart, is orthonormal (if we normalize $\sigma$ from (\ref{eq_cd_orth}) in such a way that $|\sigma(x)|_{h^{FS}} = 1$ for any $x \in \mathbb{S}^1$).
	Hence, we conclude that for any $x, y \in \mathbb{S}^1$, we have
	\begin{equation}
		B_k(x, y) = \sum_{i = 0}^{k} x^i \cdot \overline{y}^i.
	\end{equation}
	Using this formula, the reader will first observe that $B_k(x, x) = k + 1$ for all $x \in \mathbb{S}^1$, $k \in \mathbb{N}$. 
	Also, it immediately follows that for even $k$ one has $|B_k(x, -x)| = 1$. 
	As a consequence, no exponential decay occurs.  
	Moreover, one checks that for non-intersecting non-empty intervals $K_1$ and $K_2$, we have $\int_{K_1 \times K_2} \mu_k^{\mathrm{Berg}} \sim c(K_1, K_2) \cdot k^{-1}$ for some $c(K_1, K_2) > 0$.
	\par 
	Our next example will indicate that the bigness assumption on $L$ is crucial for Theorem \ref{thm_off_diag}.
	\par 
	\textbf{Example 2.}
	Consider $X := X_1 \times X_2$, where $X_1$, $X_2$ are connected projective complex manifolds of positive dimensions, and let $L := \pi_2^* L_2$, where $L_2$ is an ample line bundle over $X_2$, and $\pi_2 : X \to X_2$ is the natural projection.
	We endow $L_2$ with a smooth Hermitian metric $h^L_2$ with positive curvature, and let $h^L := \pi_2^* h^L_2$.
	Endow $X$ with a volume form $\mu$, given by the wedge product of the pull-backs of volume forms $\mu_1$ on $X_1$ and $\mu_2$ on $X_2$.
	Then the Bergman kernel $B_k(x, y)$, $x , y \in X$, associated with $\mu$ and $h^L$ will relate to the Bergman kernel $B_{k, 2}(x, y)$, $x , y \in X_2$, associated with $\mu_2$ and $h^L_2$ as follows $B_k(x, y) = B_{k, 2}(\pi_2(x), \pi_2(y)) / \int_{X_1} d \mu_1$.
	If $K_{1, 1}, K_{1, 2} \subset X_1$ are two non-intersecting compact subsets of positive volume, then for the non-intersecting compact subsets $K_1 := K_{1, 1} \times X_2, K_2 := K_{1, 2} \times X_2$ of $X$, we get 
	\begin{multline}
		\iint_{K_1 \times K_2} |B_k(x, y)|^2_{(h^L)^k} \cdot d \mu(x) \cdot d \mu (y)
		\\
		=
	 	\frac{\int_{K_{1, 1}} d \mu_1}{\int_{X_1} d \mu_1}
	 	\cdot
	 	\frac{\int_{K_{1, 2}} d \mu_1}{\int_{X_1} d \mu_1}
	 	\cdot
	 	\iint_{X_2 \times X_2} |B_k(x, y)|^2_{(h^L_2)^k} \cdot d \mu_2 (x) \cdot d \mu_2 (y).
	\end{multline}
	And since $\iint_{X_2 \times X_2} |B_k(x, y)|^2_{(h^L_2)^k} d \mu_2 (x) \cdot d \mu_2 (y) = \dim H^0(X_2, L_2^{\otimes k})$, see (\ref{eq_push_forw}), the above identity shows that Theorem \ref{thm_off_diag} doesn't hold in the above setting.
	\par 
	In our final example we point out that Theorem \ref{thm_off_diag} cannot be improved by a pointwise estimate on the off-diagonal Bergman kernel.
\par 
	\textbf{Example 3.}
	Consider a projective complex manifold $X_1$ of positive dimension $n$, endowed with an ample line bundle $L_1$.
	We endow $L_1$ with a smooth Hermitian metric $h^L_1$ of positive curvature.
	Consider a blow-up $\pi : X \to X_1$ at a point $x_1 \in X_1$, and let $L := \pi^* L_1$, $h^L := \pi^* h^L_1$.
	Endow $X_1$ with a volume form $\mu_1$ and define $\mu := \pi^* \mu_1$.
	Then the Bergman kernel $B_k(x, y)$, $x , y \in X$, associated with $\mu$ and $h^L$ will relate to the Bergman kernel $B_{k, 1}(x, y)$, $x , y \in X_1$, associated with $\mu_1$ and $h^L_1$ as $B_k(x, y) = B_{k, 1}(\pi(x), \pi(y))$.
	We see in particular that for any $x, x' \in X$ so that $\pi(x), \pi(x') = x_1$, we have $|B_k(x, x')|_{(h^L)^k} = B_{k, 1}(x_1, x_1)$.
	But by the result of Tian \cite{TianBerg}, cf. discussion after Theorem \ref{thm_diag}, $B_{k, 1}(x_1, x_1) \sim k^n$, as $k \to \infty$, and so in Theorem \ref{thm_off_diag} one cannot replace the mass of the measure by the pointwise bound on the Bergman kernel.
	
	\subsection{Localization in the Bernstein-Markov setting}\label{sect_off_diag}
	The main objective of this section is to prove Theorem \ref{thm_off_diag}, under the additional hypothesis that the measure $\mu$ satisfies the Bernstein-Markov property. We will then deduce Theorem \ref{thm_off_diag_bm} from this.
	Clearly, it suffices to establish the following result.
	\begin{thm}\label{thm_nomass}
		For any Bernstein-Markov measure $\mu$ with a non-pluripolar support, and any compact subsets $K_1, K_2 \subset X$ verifying $K_1 \cap K_2 = \emptyset$, we have
		\begin{equation}
			\lim_{k \to \infty} \int_{K_1 \times K_2} \mu_k^{\rm{Berg}} = 0.
 		\end{equation}
	\end{thm}
	\par 
	Now, quite surprisingly, our main ingredient in the proof of Theorem \ref{thm_nomass} is based on the study of the Bergman kernel \textit{on the diagonal}.
	We rely more specifically on one of the main results of Berman-Boucksom-Witt Nyström \cite{BerBoucNys} (see also Bloom-Levenberg \cite{BloomLeven1,BloomLeven2} for earlier results in this direction), which also plays a key role in the proof of Theorem \ref{thm_distr}.
	\begin{thm}[{\cite[Theorem B]{BerBoucNys}}]\label{thm_diag}
		For any Bernstein-Markov measure $\mu$ with a non-pluripolar support, the sequence of measures $\frac{1}{n_k} B_k(x, x) \cdot d \mu(x)$ on $X$ converges weakly, as $k \to \infty$, to the equilibrium measure $\mu_{\mathrm{eq}}(K, h^L)$.
	\end{thm}
	\par 
	Remark that when $\mu$ is a volume form, $L$ is ample and $h^L$ is smooth and positive, much more precise asymptotic results on the Bergman kernel can be obtained.  
	In this setting, it is known that $B_k(x, x)$ admits a full asymptotic expansion in powers of $k$, and this expansion depends smoothly on all relevant parameters -- namely, the point $x \in X$, the metric on the manifold, the volume form, the complex structure, and so on.  
	This result was established in various degrees of generality by Tian \cite{TianBerg}, Catlin \cite{Caltin}, Bouche \cite{Bouche}, Zelditch \cite{ZeldBerg}, Dai-Liu-Ma \cite{DaiLiuMa} and Ma-Marinescu \cite{MaHol}.
	\par 
	Note, however, that such a strong pointwise asymptotic expansion is highly sensitive to the underlying assumptions.  
	For instance, already if $h^L$ is only semi-positive and not strictly positive, the nature of the expansion changes significantly, see Marinescu-Savale \cite{MarinSaval}.
	Remarkably, Theorem \ref{thm_diag} tells us that the Bergman measure is much more robust in this perspective.
	\par 
	Let us now explain the relation between Theorems \ref{thm_off_diag_bm} and \ref{thm_diag}.
	We denote by $\pi: X \times X \to X$ one of the two projections. 
	Then an easy verification (see (\ref{eq_bk_id})) shows
	\begin{equation}\label{eq_push_forw}
		\pi_* \mu_k^{\rm{Berg}}
		=
		\frac{1}{n_k} B_k(x, x) \cdot d \mu(x).
	\end{equation}
	Since pushforwards under continuous maps preserve weak convergence, we observe that Theorem \ref{thm_off_diag_bm} constitutes a refinement of Theorem \ref{thm_diag}.
	\par 
	\begin{proof}[Proof of Theorem \ref{thm_nomass}]
		Assume for contradiction that there exists $\epsilon > 0$ such that, up to passing to a subsequence, for all $k \in \nat$ big enough, we have 
		\begin{equation}\label{eq_contr}
    		\int_{K_1 \times K_2} \mu_k^{\mathrm{Berg}} \geq \epsilon.
		\end{equation}
		From Theorem \ref{thm_diag}, (\ref{eq_push_forw}) and (\ref{eq_contr}), we then immediately obtain
		\begin{equation}\label{eq_non_empt}
			\int_{K_1} \mu_{\mathrm{eq}}(K, h^L)
			\geq
			\epsilon.
		\end{equation}
 		\par 
		By a version of Urysohn's lemma, we pick two continuous functions $f, g : X \to [0, 1]$, with non-intersecting support, and such that $f|_{K_1} = 1$, $g|_{K_2} = 1$.
		We consider the measure 
		\begin{equation}
			\mu' := \mu \cdot \exp(- g).
		\end{equation}
		It follows directly from the definition that $\mu'$ is a Bernstein-Markov measure for $(K, h^L)$, since $\mu$ is, and so Theorem \ref{thm_diag} holds for $\mu := \mu'$.
		Hence, if we denote by $B_k'(x, x)$, $x \in X$, $k \in \nat$, the Bergman kernel associated with $\mu'$ and $h^L$, by Theorem \ref{thm_diag}, we conclude that 
		\begin{equation}\label{eq_asympt_bk_diag}
		\begin{aligned}
			&
			\lim_{k \to \infty}
			\frac{1}{n_k}
			\int_{X} f(x) \cdot B_k'(x, x) \cdot d \mu'(x)
			=
			\int_{X} f \cdot d \mu_{\mathrm{eq}}(K, h^L), 
			\\
			&
			\lim_{k \to \infty}
			\frac{1}{n_k}
			\int_{X} f(x) \cdot B_k(x, x) \cdot d \mu(x)
			=
			\int_{X} f \cdot d \mu_{\mathrm{eq}}(K, h^L).
		\end{aligned}
		\end{equation}
		From (\ref{eq_non_empt}), we deduce that
		\begin{equation}\label{eq_non_empt2}
			\int_{X} f \cdot d \mu_{\mathrm{eq}}(K, h^L)
			\geq
			\epsilon.
		\end{equation}
		\par 
		We will now obtain a contradiction with our initial assumption (\ref{eq_contr}).
		We assume $k_0 \in \nat$ is big enough so that $n_k > 0$ for any $k \geq k_0$.
		We denote by $\Sigma$ the union of base loci of $L^{k_0}, L^{k_0 + 1}, \ldots, L^{2 k_0}$.
		By definition of the base-loci and the fact that the space of holomorphic sections of tensor powers of $L$ form an algebra under the multiplication, for any $x \in X \setminus \Sigma$ and $k \geq k_0$, there is $s \in H^0(X, L^{\otimes k})$ so that $s(x) \neq 0$.
		Remark also that $\Sigma$ is an analytic subset, and since $\mu$ is Borel, the indicator function of $\Sigma$ is measurable.
		From this, we conclude that
		\begin{equation}\label{eq_int_sigm_disap}
			\int_{X \times X} |B_k(x, y)|^2 d \mu(x) d \mu(y) = \int_{X \setminus \Sigma}  |B_k(x, y)|^2 d \mu(x) d \mu(y).
		\end{equation}
		\par 
		Now, for any $x \in X \setminus \Sigma$, we denote by $s_{k, x} \in H^0(X, L^{\otimes k})$ (resp. $s'_{k, x} \in H^0(X, L^{\otimes k})$) the \textit{peak section} at $x$ with respect to the scalar product ${\textrm{Hilb}}_k(h^L, \mu)$ (resp. ${\textrm{Hilb}}_k(h^L, \mu')$).
		Recall that this means that $s_{k, x}$ (resp. $s'_{k, x}$) is of unit norm with respect to ${\textrm{Hilb}}_k(h^L, \mu)$ (resp. ${\textrm{Hilb}}_k(h^L, \mu')$) and orthogonal to the subspace $H^0(X, L^{\otimes k} \otimes \mathcal{J}_x)$ of holomorphic sections of $L^{\otimes k}$ vanishing at $x$.
		\par 
		Immediately from the definition of peak sections and (\ref{eq_bk_id}), we deduce the following bound
		\begin{equation}\label{eq_peak_char}
			\frac{|s'_{k, x}(x)|^2_{(h^L)^k}}{\int_X |s'_{k, x}(y)|^2_{(h^L)^k} d \mu'(y)}
			\geq
			\frac{|s_{k, x}(x)|^2_{(h^L)^k}}{\int_X |s_{k, x}(y)|^2_{(h^L)^k} d \mu'(y)},
		\end{equation}
		which according to (\ref{eq_bk_id}) and our normalization yields
		\begin{equation}\label{eq_peak_char1}
			\int_X |s_{k, x}(y)|^2_{(h^L)^k} d \mu'(y)
			\geq
			\frac{B_k(x, x)}{B'_k(x, x)}.
		\end{equation}
		From (\ref{eq_bk_id}), we deduce
		\begin{multline}\label{eq_int_bk_peak}
			\int_{x \in X \setminus \Sigma} \int_{y \in X} f(x) \cdot |B_k(x, y)|^{2}_{(h^L)^k} \cdot d \mu'(x) \cdot d \mu'(y)
			\\
			=
			\int_{x \in X \setminus \Sigma} f(x) \cdot B_k(x, x) \cdot \Big( \int_{y \in X} |s_{k, x}(y)|^2_{(h^L)^k} \cdot d \mu'(y) \Big) \cdot d \mu'(x)
		\end{multline}
		By (\ref{eq_peak_char1}), we further obtain 
		\begin{multline}\label{eq_int_bk_peak2}
			\int_{x \in X \setminus \Sigma} f(x) \cdot B_k(x, x) \cdot \Big( \int_{y \in X} |s_{k, x}(y)|^2_{(h^L)^k} d \mu'(y) \Big) \cdot d \mu'(x)
			\\
			\geq
			\int_{x \in X \setminus \Sigma} f(x) \cdot \frac{B_k(x, x)^2}{B'_k(x, x)} \cdot d \mu'(x).
		\end{multline}
		By Cauchy-Schwartz inequality and the fact that $\mu'$ coincides with $\mu$ over the support of $f$,
		\begin{multline}\label{eq_int_bk_peak3}
			\Big( \int_{x \in X \setminus \Sigma} f(x) \cdot \frac{B_k(x, x)^2}{B'_k(x, x)} \cdot d \mu'(x) \Big)
			\cdot
			\Big(
			\int_{x \in X \setminus \Sigma} f(x) \cdot B'_k(x, x) \cdot d \mu'(x)
			\Big)
			\\
			\geq
			\Big( \int_{x \in X \setminus \Sigma} f(x) \cdot B_k(x, x) \cdot d \mu(x) \Big)^2.
		\end{multline}
		A combination of (\ref{eq_asympt_bk_diag}), (\ref{eq_non_empt2}), (\ref{eq_int_sigm_disap}), (\ref{eq_int_bk_peak}), (\ref{eq_int_bk_peak2}) and (\ref{eq_int_bk_peak3}) yields
		\begin{equation}\label{eq_lower_bnd}
			\liminf_{k \to \infty}
			\frac{1}{n_k}
			\int_{x \in X} \int_{y \in X} f(x) \cdot |B_k(x, y)|^{2}_{(h^L)^k} \cdot d \mu'(x) \cdot d \mu'(y)
			\geq
			\int_{X} f \cdot d \mu_{\mathrm{eq}}(K, h^L).
		\end{equation}
		\par 
		To obtain a contradiction with (\ref{eq_contr}) -- a statement we have not yet used -- we estimate from above the integral on the left-hand side of (\ref{eq_lower_bnd}).
		Immediately from the definition of $f$ and $\mu'$,
		\begin{multline}\label{eq_upp_bnd1}
			\int_{x \in X} \int_{y \in X} f(x) \cdot |B_k(x, y)|^{2}_{(h^L)^k} \cdot d \mu'(x) \cdot d \mu'(y)
			\leq
			\\
			\int_{x \in X} \int_{y \in X} f(x) \cdot |B_k(x, y)|^{2}_{(h^L)^k} \cdot d \mu(x) \cdot d \mu(y)
			\\
			-
			(1 - e^{-1}) \cdot \int_{x \in K_1} \int_{y \in K_2} |B_k(x, y)|^{2}_{(h^L)^k} \cdot d \mu(x) \cdot d \mu(y).
		\end{multline}
		From (\ref{eq_push_forw}), we deduce
		\begin{equation}\label{eq_upp_bnd2}
			\int_{x \in X} \int_{y \in X} f(x) \cdot |B_k(x, y)|^{2}_{(h^L)^k} \cdot d \mu(x) \cdot d \mu(y)
			=
			\int_{x \in X} f(x) \cdot B_k(x, x) \cdot d \mu(x).
		\end{equation}
		From (\ref{eq_contr}), (\ref{eq_asympt_bk_diag}), (\ref{eq_upp_bnd1}) and  (\ref{eq_upp_bnd2}), we obtain
		\begin{multline}
			\limsup_{k \to \infty}
			\frac{1}{n_k}
			\int_{x \in X} \int_{y \in X} f(x) \cdot |B_k(x, y)|^{2}_{(h^L)^k} \cdot d \mu'(x) \cdot d \mu'(y)
			\\
			\leq
			\int_{X} f \cdot d \mu_{\mathrm{eq}}(K, h^L) - (1 - e^{-1}) \epsilon,
		\end{multline}
		clearly contradicting (\ref{eq_lower_bnd}), and, hence, our initial assumption (\ref{eq_contr}).
		This finishes our proof.
	\end{proof}
	
	\begin{proof}[Proof of Theorem \ref{thm_off_diag_bm}]
		Let us fix $f \in \ccal^0(X \times X)$.
		We would like to verify that $\int f(x, y) \cdot d \mu_k^{\rm{Berg}}(x, y) \to \int f(x, x) \cdot d \mu_{\mathrm{eq}}(K, h^L)$, as $k \to \infty$.
		We define $g \in \ccal^0(X \times X)$ as $g(x, y) := f(x, x)$.
		Then directly from Theorem \ref{thm_diag} and (\ref{eq_push_forw}), we obtain $\int g(x, y) \cdot d \mu_k^{\rm{Berg}}(x, y) \to \int f(x, x) \cdot d \mu_{\mathrm{eq}}(K, h^L)$, as $k \to \infty$.
		By considering the difference $f - g$, we see that it suffices to show that $\int h(x, y) \cdot d \mu_k^{\rm{Berg}}(x, y) \to 0$, as $k \to \infty$, for continuous $h$ vanishing on the diagonal.
		\par 
		According to Theorem \ref{thm_nomass}, the above holds for $h$ lying in the space of the functions $\mathcal{V}$ spanned by $a(x) \cdot b(y)$ where $a, b \in \ccal^0(X)$ have non-intersecting support.
		Hence it also holds for the functions from the uniform closure $\overline{\mathcal{V}}$ of $\mathcal{V}$.
		The proof of Theorem \ref{thm_off_diag_bm} will be complete once we establish that $\overline{\mathcal{V}}$ coincides with the space of functions vanishing on the diagonal.
		\par 
		To establish this, note first that it is immediate that every function from $\overline{\mathcal{V}}$ vanishes on the diagonal.
		Also any function vanishing along the diagonal can be uniformly approximated by functions vanishing in a neighborhood of the diagonal.
		It is hence enough to show that a continuous function $h$ vanishing in a neighborhood of the diagonal lies in $\overline{\mathcal{V}}$.
		To see this, consider a partition of unity $\rho_i$, $i \in I$, subordinate to a sufficiently small mesh.
		Then from the uniform continuity of $h$, by writing $h(x, y) = \sum_{i, j \in I} h(x, y) \rho_i(x) \rho_j(y)$, one sees that the functions $\sum_{i, j \in I} h(x_{i, j}) \rho_i(x) \rho_j(y)$, where $x_{i, j} \in {\rm{supp}} (\rho_i) \times {\rm{supp}} (\rho_j)$ are chosen in an arbitrary way, approximate uniformly the function $h$ if the size of the mesh is small enough, and -- again if the mesh is small enough -- these approximations lie in $\mathcal{V}$ since $h$ vanishes in a neighborhood of the diagonal.
	\end{proof}

	\subsection{Reduction to measures supported away from subvarieties}\label{sect_local}
	This section is the first of two aimed at proving Theorem \ref{thm_off_diag} in full generality.
	We show here that if a measure $\mu$ is truncated around a subvariety, the associated off-diagonal Bergman kernel changes only slightly.
	As a consequence, to prove Theorem \ref{thm_off_diag} in full generality, it suffices to consider measures supported away from a fixed subvariety.
	\par 
	It turns out that this form of “continuity under truncation” extends beyond subvarieties to a broader class -- namely, close subsets that are not charged by any subsequential limit of the diagonal Bergman measures.
	The reason why this “continuity" applies to the analytic subvarieties then follows from the following result that we shall establish in Section \ref{sect_bergm_anal}.
	\begin{thm}\label{thm_anal_no_mass}
		For any Borel measure $\mu$ which does not give full mass to pluripolar subsets, any subvariety $Y \subset X$ and any $\epsilon > 0$, there exists an open neighborhood $U$ of $Y$ and $k_0 \in \nat$, so that for any $k \geq k_0$, we have
		\begin{equation}\label{eq_anal_no_mass}
			\int_U \mu_k^{\rm{Berg}, \Delta} \leq \epsilon.
		\end{equation}
	\end{thm}
	\par 
	To formulate a general form of the continuity principle, we fix a closed subset $A \subset X$ and consider a decreasing sequence of open neighborhoods $U_1 \Supset U_2 \Supset \cdots$ of $A$ such that $\bigcap_{i = 1}^{+\infty} U_i = A$. 
	Such a sequence always exists since $A$ is closed.
	By Urysohn's lemma, we consider bump functions $\rho_i : X \to [0, 1]$, $i \in \nat^*$, verifying $\rho_i = 1$ over $X \setminus U_i$ and $\rho_i = 0$ over $U_{i + 1}$. 
	We assume that for any $\epsilon > 0$, there is $i \in \nat^*$, so that for any $k \in \nat$, we have
	\begin{equation}\label{eq_anal_no_mass_a}
		\int_{U_i} \mu_k^{\rm{Berg}, \Delta} \leq \epsilon. 
	\end{equation}
	\par 
	We denote $\mu_i := \rho_i \cdot \mu$, and let $B_{k, i}(x, y) \in L^k_x \otimes (L^k_y)^*$, be the Bergman kernel associated with $h^L$ and $\mu_i$. 
	We denote by $\mu_{k, i}^{\rm{Berg}}$ the associated probability measures on $X \times X$, defined as in (\ref{eq_mu_begmm}).
	The main result of this section goes as follows.
	\begin{prop}\label{prop_reduc}
		Assume that a compact subset $K \subset X \times X$ is such that for any $\epsilon > 0$, $i \in \nat$, there is $k_i \in \nat$, so that $\int_K \mu_{k, i}^{\rm{Berg}} \leq \epsilon$, for any $k \geq k_i$.
		Then for any $\epsilon > 0$, there is $k_0 \in \nat$, so that $\int_K \mu_{k}^{\rm{Berg}} \leq \epsilon$, for any $k \geq k_0$.
	\end{prop}
	In order to establish Proposition \ref{prop_reduc}, the following lemma will be crucial.
	\begin{lem}\label{lem_local_regul}
		For any $\epsilon > 0$, there are $k_0, i_0 \in \nat$, so that for any $k \geq k_0$ and $i \geq i_0$, we have
		\begin{equation}
			 \iint \Big| B_{k, i}(x, y) - B_k(x, y) \Big|_{(h^L)^k}^2 \cdot d \mu_i(x) d \mu_i(y)
			 \leq
			 \epsilon \cdot n_k.
		\end{equation}
	\end{lem}
	The proof of Lemma \ref{lem_local_regul} is based on the following comparison result on the off-diagonal Bergman kernels due to Lubinsky \cite[(3.3)]{LubinskyUnivers}.
	\begin{lem}\label{lem_lub}
		Assume that the measures $\mu_1$ and $\mu_2$ are ordered as $\mu_1 \leq \mu_2$.
		Then for any $x \in X$, the associated Bergman kernels are related as
		\begin{equation}\label{eq_lem_lub}
			\int
			\Big|
			B_{k, 1}(x, y)
			-
			B_{k, 2}(x, y)
			\Big|_{(h^L)^k}^2 
			\cdot
			d \mu_1(y)
			\leq
			B_{k, 1}(x, x)
			-
			B_{k, 2}(x, x).
		\end{equation}
	\end{lem}
	\begin{rem}\label{rem_lub}
		Observe that the characterization of the Bergman kernel analogous to (\ref{eq_bergman_max}) directly ensures that the right-hand side of (\ref{eq_lem_lub}) is positive.
	\end{rem}
	\begin{proof}
		Lubinsky in \cite{LubinskyUnivers} establishes (\ref{eq_lem_lub}) only in the setting of measures on the real line, but the argument relies exclusively on the reproducing property of the Bergman kernel, and so it remains valid in our generality.
		For the convenience of the reader, we reproduce the argument below.
		\par 
		Using the fact that the Bergman kernel is self-adjoint, we write
		\begin{equation}
			\Big|
			B_{k, 1}(x, y)
			-
			B_{k, 2}(x, y)
			\Big|_{(h^L)^k}^2 
			=
			\Big(
			B_{k, 1}(x, y)
			-
			B_{k, 2}(x, y)
			\Big)
			\cdot
			\Big(
			B_{k, 1}(y, x)
			-
			B_{k, 2}(y, x)
			\Big).
		\end{equation}
		Now, by the reproducing property of the Bergman kernel, we have
		\begin{equation}
		\begin{aligned}
			&
			\int
			B_{k, 1}(x, y)
			\cdot
			B_{k, 1}(y, x)
			\cdot
			d \mu_1(y)
			=
			B_{k, 1}(x, x),
			\\
			&
			\int
			B_{k, 1}(x, y)
			\cdot
			B_{k, 2}(y, x)
			\cdot
			d \mu_1(y)
			=
			B_{k, 2}(x, x).
		\end{aligned}
		\end{equation}
		By taking the adjoint of the second equation above, we have
		\begin{equation}
			\int
			B_{k, 2}(x, y)
			\cdot
			B_{k, 1}(y, x)
			\cdot
			d \mu_1(y)
			=
			B_{k, 2}(x, x).
		\end{equation}
		Finally, by using the bound $\mu_1 \leq \mu_2$, we obtain
		\begin{equation}
			\int
			|B_{k, 2}(x, y)|_{(h^L)^k}^2 
			\cdot
			d \mu_1(y)
			\leq
			\int
			|B_{k, 2}(x, y)|_{(h^L)^k}^2 
			\cdot
			d \mu_2(y),
		\end{equation}
		which we rewrite using the reproducing property of the Bergman kernel and the fact that the Bergman kernel is self-adjoint as
		\begin{equation}
			\int
			B_{k, 2}(x, y)
			\cdot
			B_{k, 2}(y, x)
			\cdot
			d \mu_1(y)
			\leq
			B_{k, 2}(x, x).
		\end{equation}
		A combination of the above estimates yield (\ref{eq_lem_lub}).
	\end{proof}
	\begin{proof}[Proof of Lemma \ref{lem_local_regul}]
		Note first that immediately from our construction of $\rho_i$, we have $\mu_1 \leq \mu_2 \leq \cdots \leq \mu$.
		From this and Lemma \ref{lem_lub}, we deduce that for any $i \in \nat$, we have
		\begin{equation}\label{lem_local_regul1}
			\iint
			\Big|
			B_{k, i}(x, y)
			-
			B_{k}(x, y)
			\Big|_{(h^L)^k}^2 
			\cdot
			d \mu_i(x)
			d \mu_i(y)
			\leq
			\int
			\Big(
			B_{k, i}(x, x)
			-
			B_{k}(x, x)
			\Big)
			\cdot
			d \mu_i(x).
		\end{equation}
		By our definition of $\mu_i$, we have the following bound
		\begin{equation}\label{lem_local_regul2}
			\int
			B_{k}(x, x)
			\cdot
			d \mu_i(x)
			\geq
			\int_{X \setminus U_i}
			B_{k}(x, x)
			\cdot
			d \mu(x).
		\end{equation}
		Remark, however, that we have
		\begin{equation}\label{lem_local_regul4}
			\int
			B_{k, i}(x, x)
			\cdot
			d \mu_i(x) = n_k, 
		\end{equation}
		and similarly for the Bergman kernel associated with $\mu$.
		By this and (\ref{eq_anal_no_mass_a}), for any $\epsilon > 0$, there are $i_0, k_0 \in \nat$, so that for any $i \geq i_0$, $k \geq k_0$,
		\begin{equation}\label{lem_local_regul3}
			\int_{X \setminus U_i}
			B_{k}(x, x)
			\cdot
			d \mu(x)
			\geq
			(1-\epsilon)n_k.
		\end{equation}
		Lemma \ref{lem_local_regul} now follows immediately from (\ref{lem_local_regul1}), (\ref{lem_local_regul2}), (\ref{lem_local_regul4}) and (\ref{lem_local_regul3}).
	\end{proof}
	\begin{proof}[Proof of Proposition \ref{prop_reduc}]
		Directly from (\ref{eq_anal_no_mass_a}), we see that for any $\epsilon > 0$, there are $i_0, k_0 \in \nat$, so that for any $i \geq i_0$, $k \geq k_0$, we have
		\begin{equation}\label{eq_prop_reduc1}
			\iint_{U_i \times X} \big| B_k(x, y) \big|_{(h^L)^k}^2 \cdot d \mu(x) d \mu(y)
			\leq
			\frac{\epsilon \cdot n_k}{4}.
		\end{equation}
		Moreover, from Lemma \ref{lem_local_regul}, our assumption $\int_K \mu_{k, i}^{\rm{Berg}} \leq \epsilon$, and Minkowski inequality, we see that for any $\epsilon > 0$, there are $k_0 \in \nat$, $i \geq i_0$, so that for any $k \geq k_0$, we have
		\begin{equation}\label{eq_prop_reduc2}
			 \iint_{K} \big| B_k(x, y) \big|_{(h^L)^k}^2 \cdot d \mu_i(x) d \mu_i(y)
			 \leq
			 \frac{\epsilon \cdot n_k}{2}.
		\end{equation}
		Note, however, that by the definition of $\mu_i$, we have
		\begin{multline}\label{eq_prop_reduc3}
			 \iint_{K} \big| B_k(x, y) \big|_{(h^L)^k}^2 \cdot d \mu_i(x) d \mu_i(y)
			 \geq
			 \iint_{K} \big| B_k(x, y) \big|_{(h^L)^k}^2 \cdot d \mu(x) d \mu(y)
			 \\
			 -
			 2 \cdot \iint_{U_i \times X} \big| B_k(x, y) \big|_{(h^L)^k}^2 \cdot d \mu(x) d \mu(y).
		\end{multline}
		A combination of (\ref{eq_prop_reduc1}), (\ref{eq_prop_reduc2}) and (\ref{eq_prop_reduc3}) yields Proposition \ref{prop_reduc}.
	\end{proof}
	
	\subsection{Decay of the kernel for measures supported away from a divisor}\label{sect_supp}
	This section is the second of two devoted to proving Theorem \ref{thm_off_diag}.
	Here we show that for measures supported away from certain divisors, a quantitative version of Theorem \ref{thm_off_diag} holds.
	\par 
	We will now assume that $k_0 \in \nat$ is big enough so that $n_{k_0} \geq 2$.
	We fix two linearly independent sections $s_1, s_2 \in H^0(X, L^{\otimes k_0})$, and consider the subvariety $Y := {\textrm{div}}(s_2)$.
	We fix a Borel measure $\mu_0$ on $X$ whose support is disjoint from $Y$.
	If the support of $\mu_0$ is not an analytic subset, then the form defined as in (\ref{eq_defn_l2}), but for $\mu_0$ instead of $\mu$, is positive definite, and hence the associated Bergman kernel, $B_{k, 0}(x, y) \in L^k_x \otimes (L^k_y)^*$, can be defined.
	In general, one can define $B_{k, 0}(x, y)$ by considering an orthonormal basis of a vector subspace $E_k \subset H^0(X, L^{\otimes k})$ of maximal dimension, so that the restriction map $E_k \to L^2(\mu)$ is injective.
	While $|B_{k, 0}(x, y)|^2_{(h^L)^k}$ might depend on the choice of $E_k$, an easy verification shows that the resulting measure $|B_{k, 0}(x, y)|^2_{(h^L)^k} d \mu_0(x) d \mu_0 (y)$ is independent of it.
	The principal result of this section — obtained without requiring any condition on $\mu_0$ other than that its support be disjoint from $Y$ — is the following.
	\begin{thm}\label{thm_qual_deloc}
		There is $C > 0$, so that for any $k \in \nat^*$, we have
		\begin{equation}\label{eq_thm_qual_deloc}
			\iint \big|B_{k, 0}(x, y)\big|^2_{(h^L)^k} \cdot \Big| \frac{s_1(x)}{s_2(x)} - \frac{s_1(y)}{s_2(y)} \Big|^2 \cdot d \mu_0(x) d \mu_0 (y)
			\leq 
			C \cdot (n_k - n_{k - k_0}).
		\end{equation}
	\end{thm}
	\begin{rem}\label{rem_fuj}
		a) Note that by Fujita's theorem \cite{Fujita}, recalled after (\ref{defn_vol}), the right-hand side of (\ref{eq_thm_qual_deloc}) is asymptotically smaller than $n_k$.
		In particular, if for a compact subset $K \subset (X \setminus Y) \times (X \setminus Y)$, there exists a constant $C_1 > 0$ such that $| s_1(x)/s_2(x) - s_1(y)/s_2(y)| \geq C_1$, for any $x \times y \in K$, then Theorem \ref{thm_qual_deloc} provides a quantitative version of Theorem \ref{thm_off_diag}.
		\par 
		b)
		In the setting of orthogonal polynomials described in (\ref{eq_cd_orth}), one can consider $s_2 := \sigma^2$ and $s_1 = z \cdot \sigma^2$, where $\sigma \in H^0(\mathbb{P}^1, \mathscr{O}(1))$ is as described after (\ref{eq_cd_orth}).
		Then for any compact $K \subset \comp \times \comp \subset \mathbb{P}^1 \times \mathbb{P}^1$, not intersecting the diagonal, the above $C_1$ exists. 
		We conclude that for arbitrary compactly supported measures $\mu$ on $\comp$ with non-finite support, the estimate (\ref{eq_rem_fuj}) holds.
		We note that when the support of $\mu$ lies inside of the real axis, one can give a shorter proof of (\ref{eq_rem_fuj}) based on the Christoﬀel-Darboux formula, see in particular \cite[(2.13)]{NevaiCond}.
	\end{rem}
	\begin{proof}
		Our proof is inspired by the the proof of \cite[Theorem 17]{BermanOrtega} by Berman-Ortega-Cerdá in the sense that the estimate (\ref{eq_thm_qual_deloc}) is established by studying the Hilbert-Schmidt norm of a certain sequence of operators.
		However, in contrast to \cite{BermanOrtega}, our operators are not of Toeplitz type. 
		This allows us to bypass -- at the expense of a more delicate analysis of the associated Schwartz kernel -- the assumption present in \cite{BermanOrtega} that the support of the measure is totally real.
		A somewhat similar sequence of operators was used in the context of orthogonal polynomials by Simon \cite{SimonWeakConv}.
		\par 
		More precisely, we let $f := s_1/s_2$ and consider a sequence of operators $S_k: L^2(X, L^{\otimes k}, \mu_0) \to L^2(X, L^{\otimes k}, \mu_0)$, defined through their Schwartz kernels
		\begin{equation}\label{eq_sk_sch1}
			S_k(x, y) := \int B_{k, 0}(x, z) \cdot \big( f(x) - f(z) \big) \cdot B_{k, 0}(z, y) \cdot d \mu_0(z).
		\end{equation}
		\par 
		Let us verify that the operators $S_k$, $k \in \nat$, are uniformly bounded, i.e. there is $C > 0$, so that $\| S_k \| \leq C$, where $\| \cdot \|$ is the operator norm.
		Since the measure $\mu_0$ is supported away from $Y$, the function $f$ is uniformly bounded on its support.
		The operator $S_k$ is given by a composition of a Bergman projection (which is of unit norm) and the commutator of the Bergman projection and the multiplication operator by $f$ (which is of bounded norm, independent of $k \in \nat$).
		Hence, there is $C > 0$, so that $\| S_k \| \leq C$, for any $k \in \nat$.
		\par 
		Now, let us verify that the kernel of $S_k$ has codimension at most $n_k - n_{k - k_0}$.
		Since the second factor of $S_k$ is given by the Bergman projection, we see that $H^0(X, L^{\otimes k})^{\perp} \subset \ker S_k$.
		Now, consider a sequence of linear operators $F_k : H^0(X, L^{\otimes (k - k_0)}) \to H^0(X, L^{\otimes k})$, given by a multiplication by $s_2$.
		It is immediate to verify that $\Im F_k \subset \ker S_k$.
		But the map $F_k$ is clearly an injection, and so the dimension of its image is $n_{k - k_0}$.
		This along with the obvious fact $\Im F_k \cap H^0(X, L^{\otimes k})^{\perp} = \emptyset$ yields the needed bound on the codimension of the kernel of $S_k$.
		\par 
		As a conclusion, we deduce the following bound
		\begin{equation}\label{eq_sk_sch2}
			\| S_k \|_{{\textrm{HS}}}^2 \leq C \cdot (n_k - n_{k - k_0}),
		\end{equation}
		where $\| \cdot \|_{{\textrm{HS}}}$ is the Hilbert-Schmidt norm.
		Since the Hilbert-Schmidt norm of an integral operator can be expressed through the $L^2$-norm of its Schwartz kernel, we deduce
		\begin{equation}\label{eq_sk_sch3}
			\iint \big| S_k(x, y) \big|^2_{(h^L)^k} \cdot d \mu_0(x) d \mu_0(y) \leq  C \cdot (n_k - n_{k - k_0}).
		\end{equation}
		\par 
		\begin{sloppypar}
		We shall now see that (\ref{eq_sk_sch3}) is essentially a restatement of (\ref{eq_thm_qual_deloc}).
		By using the identity $\big| S_k(x, y) \big|^2_{(h^L)^k} = S_k(x, y) \cdot S_k(x, y)^*$ and (\ref{eq_sk_sch1}), we rewrite 
		\begin{multline}\label{eq_sk_sch4}
			\iint \big| S_k(x, y) \big|^2_{(h^L)^k} \cdot  d \mu_0(x) d \mu_0(y) 
			\\
			=
			\iint  \int B_{k, 0}(x, z) \cdot \big( f(x) - f(z) \big) \cdot B_{k, 0}(z, y) \cdot d \mu_0(z) \cdot
			\\
			\cdot \int  B_{k, 0}(x, w)^* \cdot \big( \overline{f}(x) - \overline{f}(w) \big) \cdot B_{k, 0}(w, y)^*  \cdot  d \mu_0(w)  \cdot  d \mu_0(x) d \mu_0(y).
		\end{multline}
		\end{sloppypar}
		\noindent
		Using Fubini's theorem and the symmetry $B_{k, 0}(x, y)^* = B_{k, 0}(y, x)$, we further rewrite
		\begin{multline}\label{eq_sk_sch5}
			\iint \big| S_k(x, y) \big|^2_{(h^L)^k} \cdot  d \mu_0(x) d \mu_0(y) 
			\\
			=
			\iiiint B_{k, 0}(x, z) B_{k, 0}(z, y) B_{k, 0}(w, x) B_{k, 0}(y, w) \cdot 
			\\
			\cdot \big( f(x) - f(z) \big) \cdot \big( \overline{f}(x) - \overline{f}(w) \big) \cdot  d \mu_0(z) d \mu_0(w) d \mu_0(x) d \mu_0(y).
		\end{multline}
		We expand the integrand, apply the reproducing property of the Bergman kernel to eliminate two of the four kernel factors in each of the resulting terms as follows
		\begin{equation}
		\begin{aligned}
			&\iint B_{k, 0}(x, z) B_{k, 0}(z, y) B_{k, 0}(w, x) B_{k, 0}(y, w) \cdot f(x) \overline{f}(x) \cdot  d \mu_0(z) d \mu_0(w)
			\\
			&
			\qquad \qquad \qquad \qquad \qquad \qquad\qquad \qquad \qquad \qquad \qquad
			= \big| B_{k, 0}(x, y) \big|^2_{(h^L)^k} \cdot f(x) \overline{f}(x),	
			\\
			&\iint B_{k, 0}(x, z) B_{k, 0}(z, y) B_{k, 0}(w, x) B_{k, 0}(y, w) \cdot f(x) \overline{f}(w) \cdot  d \mu_0(z) d \mu_0(y)
			\\
			&
			\qquad \qquad \qquad \qquad \qquad \qquad\qquad \qquad \qquad \qquad \qquad
			= \big| B_{k, 0}(x, w) \big|^2_{(h^L)^k} \cdot f(x) \overline{f}(w),		
			\\
			&\iint B_{k, 0}(x, z) B_{k, 0}(z, y) B_{k, 0}(w, x) B_{k, 0}(y, w) \cdot f(z) \overline{f}(x) \cdot  d \mu_0(w) d \mu_0(y)
			\\
			&
			\qquad \qquad \qquad \qquad \qquad \qquad\qquad \qquad \qquad \qquad \qquad
			= \big| B_{k, 0}(x, z) \big|^2_{(h^L)^k} \cdot f(z) \overline{f}(x),		
			\\
			&\iint B_{k, 0}(x, z) B_{k, 0}(z, y) B_{k, 0}(w, x) B_{k, 0}(y, w) \cdot f(z) \overline{f}(w) \cdot  d \mu_0(x) d \mu_0(y)
			\\
			&
			\qquad \qquad \qquad \qquad \qquad \qquad\qquad \qquad \qquad \qquad \qquad
			= \big| B_{k, 0}(z, w) \big|^2_{(h^L)^k} \cdot f(z) \overline{f}(w).	
		\end{aligned}
		\end{equation}
		By renaming the variables, we obtain
		\begin{multline}\label{eq_sk_sch6}
			\iint \big| S_k(x, y) \big|^2_{(h^L)^k}  \cdot  d \mu_0(x) d \mu_0(y) 
			\\
			=
			\frac{1}{2}
			\iint \big| B_{k, 0}(x, y) \big|^2_{(h^L)^k} \cdot \big| f(x) - f(y) \big|^2 \cdot d \mu_0(x) d \mu_0(y),
		\end{multline}
		which finishes the proof of Theorem \ref{thm_qual_deloc} by (\ref{eq_sk_sch3}).
	\end{proof}
	
	We are now finally in a position to prove the main result of this paper.
	\begin{proof}[Proof of Theorem \ref{thm_off_diag}]
		\begin{sloppypar}
		Recall that by Iitaka fibration theorem, cf. \cite[Theorem 2.1.33 and Definition 2.2.1]{LazarBookI}, our bigness assumption on $L$ assures that for $k$ big enough, the Kodaira map $X \dashrightarrow \mathbb{P}(H^0(X, L^{\otimes k}))$ is birational onto its image.
		Hence, the subset $\Sigma := \{ x \in X : \text{ there is } y \in X \setminus \{x\}, \text{ verifying } s/s'(x) = s/s'(y) \text{ for any } s, s' \in H^0(X, L^{\otimes k}) \}$ is a proper subvariety of $X$.
		\end{sloppypar}
		\par 
		We now consider a subset $Y := \Sigma \cup {\textrm{div}}(s_1) \cup \cdots \cup {\textrm{div}}(s_{n_k})$, where $s_i$, $i = 1, \ldots, n_k$, is a fixed basis of $H^0(X, L^{\otimes k})$.
		The set $Y$ is a subvariety.
		Hence, according to Theorem \ref{thm_anal_no_mass} and Proposition \ref{prop_reduc}, it suffices to establish Theorem \ref{thm_off_diag} for measures $\mu$ which are supported away from $Y$.
		Note, however, that for such measures $\mu$, for a fixed $K \subset X \times X$ not intersecting the diagonal, there is $C > 0$, so that for $K' := K \cap {\rm{supp}}(\mu) \times {\rm{supp}}(\mu)$, for any $x, y \in K'$, we have
		\begin{equation}\label{eq_low_bnd}
			\sum_{\substack{i, j = 1 \\ i \neq j}}^{n_k}
			\Big| \frac{s_i(x)}{s_j(x)} - \frac{s_i(y)}{s_j(y)} \Big|^2
			\geq 
			C.
		\end{equation}		
		Theorem \ref{thm_off_diag} for such measures $\mu$ now follows immediately from Theorem \ref{thm_qual_deloc}, Remark \ref{rem_fuj} and (\ref{eq_low_bnd}).
		As explained above, this finishes our proof.
	\end{proof}
	
	\section{On the weak limit of the diagonal Bergman measures}\label{sect_cd_diag}
	This section is mostly devoted to the following problem: which measures arise as subsequential limits of diagonal Bergman measures?
	As a consequence of our study, we will derive Theorem \ref{thm_anal_no_mass} in Section \ref{sect_bergm_anal}, modulo a statement that is established in Section \ref{sect_forbid}.
	In Theorem \ref{sect_bergm_pp}, we then present a generalization of Theorem \ref{thm_anal_no_mass} and discuss a related open problem.
	Finally, Section \ref{sect_loc_diag} is devoted to the proof of Theorem \ref{thm_diag_local}.
	We retain the notation from Introduction and assume throughout the rest of the section that the measure $\mu$ does not give full mass to pluripolar subsets.
	
	\subsection{Diagonal Bergman measures put no mass on subvarieties}\label{sect_bergm_anal}
	\begin{sloppypar} 
	The primary objective of this section is to demonstrate that every subsequential limit of the diagonal Bergman measures vanishes on subvarieties, i.e. to show Theorem \ref{thm_anal_no_mass}. 
	Remark that when $\mu$ is Bernstein-Markov with non-pluripolar support, Theorem \ref{thm_anal_no_mass} follows from Theorem \ref{thm_diag}, as equilibrium measure puts no mass on pluripolar subsets (in particular on analytic subsets).
	In order to prove Theorem \ref{thm_anal_no_mass} in full generality, we consider the so-called partial Bergman kernels associated with $Y$, denoted by  $B_k^{\epsilon, Y}(x, y) \in L^k_x \otimes (L^k_y)^*$ for $\epsilon > 0$, and defined as
	\begin{equation}\label{eq_bergm_kern_part}
		B_k^{\epsilon, Y}(x, y) := \sum_{i = 1}^{n_k(\epsilon)} s_i(x) \cdot s_i(y)^*,
	\end{equation}
	where $s_i$, $i = 1, \ldots, n_k(\epsilon)$ is an orthonormal basis of $(H^0(X, L^{\otimes k} \otimes \mathcal{J}_{Y}^{\lceil \epsilon k \rceil}), {\textrm{Hilb}}_k(h^L, \mu))$, $n_k(\epsilon) := \dim H^0(X, L^{\otimes k} \otimes \mathcal{J}_{Y}^{\lceil \epsilon k \rceil})$, and $\mathcal{J}_{Y}$ is the sheaf of holomorphic sections vanishing along $Y$.
	The following result, out of independent interest, will be established in Section \ref{sect_forbid}.
	\end{sloppypar}
	\par 
	\begin{thm}\label{thm_forbid_region}
		For any $\epsilon, \delta > 0$, there is an open neighborhood $U$ of $Y$, so that for any $x \in U$, $k \in \nat$, we have $B_k^{\epsilon, Y}(x, x) \leq \exp(- \delta k)$.
	\end{thm}
	\begin{rem}
		When $\mu$ is a volume form, Theorem \ref{thm_forbid_region} follows from more precise estimates of Berman \cite[Lemma 4.1]{BermanEnvProj}.
	\end{rem}
	\begin{proof}[Proof of Theorem \ref{thm_anal_no_mass}]
		We denote by ${\rm vol}(L)$ the volume of a line bundle $L$ on $X$, defined by
		\begin{equation}\label{defn_vol}
			{\rm{vol}}(L) := \limsup_{k \to \infty} \frac{n!}{k^n} \dim H^0(X, L^{\otimes k}).
		\end{equation}
		By Fujita's theorem \cite{Fujita}, the above upper limit is actually a limit.
		The definition of the volume functional naturally extends to the $\mathbb{Q}$-line bundles, and the resulting functional extends continuously to the full Néron-Severi space $NS^1(X)$, cf. \cite[Theorem 2.2.44]{LazarBookI}.
		\par 
		\begin{sloppypar}
		Remark also that for any $l \in \nat$, there is an isomorphism 
		\begin{equation}\label{eq_isom_blow}
			H^0(\hat{X}, \pi^* L^{\otimes k} \otimes \mathscr{O}(- l \cdot E)) \to H^0(X, L^{\otimes k} \otimes \mathcal{J}_{Y}^l),
		\end{equation}
		where $\pi : \hat{X} \to X$ is the blow-up along the ideal sheaf $\mathcal{J}_{Y}$ of holomorphic functions vanishing along $Y$, and $\mathscr{O}_{\hat{X}}(-E) := \pi^* \mathcal{J}_{Y}$.
		In particular, the limit $\lim_{k \to \infty} \frac{n!}{k^n} \dim H^0(X, L^{\otimes k} \otimes \mathcal{J}_{Y}^{\lceil \epsilon' k \rceil})$ exists and can be interpreted as the volume (on $\hat{X}$) or the $\mathbb{R}$-line bundle $L \otimes \mathscr{O}_{\hat{X}}(-E)^{\epsilon'}$.
		\end{sloppypar}
		\par 
		From this and the continuity of the volume, for a given $\epsilon > 0$, we can find $\epsilon' > 0$, so that 
		\begin{equation}\label{eq_vol_cont}
			\liminf_{k \to \infty} \frac{\dim H^0(X, L^{\otimes k} \otimes \mathcal{J}_{Y}^{\lceil \epsilon' k \rceil})}{\dim H^0(X, L^{\otimes k})}
			\geq 1 - \frac{\epsilon}{2}.
		\end{equation}
		\par 
		By choosing $\epsilon := \epsilon'$ and $\delta := 1$ in Theorem \ref{thm_forbid_region}, we get a neighborhood $U$ of $Y$, so that for any $x \in U$, $k \in \nat$, we have 
		\begin{equation}\label{eq_exp_bnd}
				B_k^{\epsilon', Y}(x, x) \leq \exp(- k).
		\end{equation}
		By writing $B_k(x, x) = \big(B_k(x, x) - B_k^{\epsilon', Y}(x, x)\big) + B_k^{\epsilon', Y}(x, x)$, we obtain
		\begin{equation}\label{eq_int_berg_diag_coh110}
			\int_U B_k(x, x) \cdot d \mu(x) \leq \int_U \big(B_k(x, x) - B_k^{\epsilon', Y}(x, x)\big)   \cdot d \mu(x) + \exp(-k).
		\end{equation}
		Due to the positivity of the integrand, we have
		\begin{equation}
		 	\int_U \big(B_k(x, x) - B_k^{\epsilon', Y}(x, x)\big)   \cdot d \mu(x) 
		 	\leq
		 	 \int_X \big(B_k(x, x) - B_k^{\epsilon', Y}(x, x)\big)   \cdot d \mu(x) 
		\end{equation}
		Using the elementary identities
		\begin{equation}\label{eq_int_berg_diag_coh}
		\begin{aligned}
		& \int B_k(x, x) \cdot d\mu(x) = \dim H^0(X, L^{\otimes k}),
		\\ 
		& \int B_k^{\epsilon', Y}(x, x) \cdot d\mu(x) =  \dim H^0(X, L^{\otimes k} \otimes \mathcal{J}_{Y}^{\lceil \epsilon' k \rceil}),
		\end{aligned}
		\end{equation}
		together with (\ref{eq_vol_cont}), we however deduce
		\begin{equation}\label{eq_int_berg_diag_coh111}
			 \int_X \big(B_k(x, x) - B_k^{\epsilon', Y}(x, x)\big)   \cdot d \mu(x) 
			 \leq
			 \epsilon \cdot n_k.
		\end{equation}		 
		A combination of (\ref{eq_int_berg_diag_coh110}) and (\ref{eq_int_berg_diag_coh111}) yields (\ref{eq_anal_no_mass}).
	\end{proof}

	\subsection{Forbidden region for the partial Bergman kernels}\label{sect_forbid}
	The main goal of this section is to establish that the partial Bergman kernels decrease exponentially in the neighborhood of the subvariety, that is, to prove Theorem \ref{thm_forbid_region}.
	\par 
	To simplify the presentation, we claim that it suffices to establish Theorem \ref{thm_forbid_region} under the assumption that $\mathcal{J}_Y = \mathscr{O}_{\hat{X}}(-E)$ for some effective divisor $E$ on $X$.
	To see this, let $\pi : \hat{X} \to X$ denote the blow-up of $X$ along $\mathcal{J}_Y$, and let $\mathscr{O}_{\hat{X}}(-E) := \pi^* \mathcal{J}_Y$.  
	If $\hat{X}$ is not smooth, we perform further blow-ups to resolve its singularities.	
	By an abuse of notation, we denote the resolution by $\hat{X}$.
	Let $\hat{\mu}$ be a measure on $\hat{X}$ such that $\pi_* \hat{\mu} = \mu$.  
	Since $\mu$ does not give full mass to pluripolar subsets and $\pi$ is a biholomorphism outside an analytic (hence pluripolar) subset, the measure $\hat{\mu}$ does not give full mass to pluripolar subsets.
	The isomorphism (\ref{eq_isom_blow}) is then an isometry with respect to the induced $L^2$-norms.  
	In particular, the associated diagonal Bergman measures $\hat{\mu}_k^{\mathrm{Berg}, \Delta}$ on $\hat{X}$ and $\mu_k^{\mathrm{Berg}, \Delta}$ on $X$ satisfy the compatibility relation $\pi_* \hat{\mu}_k^{\mathrm{Berg}, \Delta} = \mu_k^{\mathrm{Berg}, \Delta}$.  
	It follows that it is enough to prove Theorem \ref{thm_forbid_region} on $\hat{X}$.
	This allows us to assume from now on that $\mathcal{J}_Y = \mathscr{O}_{\hat{X}}(-E)$ for some effective divisor $E$ on $X$.
	\par 
	Our proof of Theorem \ref{thm_forbid_region} will be based on the following three auxiliary statements.
	For the first one, we use the following notation: for a subset $K \subset X$, we denote by ${\textrm{Ban}}_k^{\infty}(K, h^L)$ the $\sup$-seminorm on $H^0(X, L^{\otimes k})$ induced by $h^L$ and evaluated over $K$. 
	For brevity, we denote ${\textrm{Ban}}_k^{\infty}(X, h^L)$ by ${\textrm{Ban}}_k^{\infty}(h^L)$.
	\begin{prop}\label{lem_asympt_bm}
		There is a continuous metric $h^L_0$ on $L$, so that for any $k \in \nat$, we have
		\begin{equation}\label{eq_asympt_bm}
			{\textrm{Hilb}}_k(h^L, \mu)
			\geq
			{\textrm{Ban}}_k^{\infty}(h^L_0).
		\end{equation}
	\end{prop}
	For non-pluripolar measures $\mu$, Proposition \ref{lem_asympt_bm} was established in \cite[Theorem~7.1]{FinEigToepl}, 
where the author also provided an explicit description of $h^L_0$ in terms of the non-negligible psh envelope from \cite{GuedjLuZeriahEnv}.  
	An analogous explicit description can be given in our setting when $\mu$ does not give full mass to pluripolar subsets, but we decided to omit the details here for brevity.
	\begin{proof}
		We work on the level of potentials instead of metrics. 
		Without loosing the generality, we can assume that $h^L$ is smooth.
		We then denote $\omega := 2 \pi c_1(L, h^L)$.
		We shall establish the following result: there is $C > 0$, such that for any $\psi \in {\rm{PSH}}(X, \omega)$, $p \geq 1$, 
		\begin{equation}\label{bm_weights}
			\sup_X \big( \exp(p \cdot \psi) \big)
			\leq
			\exp(C p)
			\cdot
			\int_X \exp(p \cdot \psi) d \mu.
		\end{equation}
		Once we establish (\ref{bm_weights}), by plugging in $\psi := \frac{1}{k} \log |s|_{h^{L^{\otimes k}}}$ and $p = 2 k$ in (\ref{bm_weights}), we would get (\ref{eq_asympt_bm}) for $h^L_0 := h^L \cdot \exp(- C)$.
		\par 
		We normalize $\mu$ so that it becomes a probability measure.
		For $p \in ]0, +\infty[$, we define
		\begin{equation}
			F_p(\psi) := \frac{1}{p} \log \int_X \exp(p \cdot \psi ) d \mu, 
			\quad 
			F(\psi) := \sup_X (\psi),
			\quad 
			\text{for any }\psi \in {\rm{PSH}}(X, \omega).
		\end{equation}
		The generalized mean inequality implies that $F_p$ is a non-decreasing function of $p$.
		Hence, to establish (\ref{bm_weights}), it would suffice to show that there is $C \in \real$, so that 
		\begin{equation}\label{eq_bnd_low}
			F_1(\psi) - F(\psi) \geq C, \quad \text{for any } \psi \in {\rm{PSH}}(X, \omega).
		\end{equation}
		\par 
		Remark now that $F_1 - F$ is clearly invariant by translation (adding a constant to the parameter), thus it descends on a function on the quotient space, which is isomorphic with the space of positive $(1, 1)$-currents lying in the cohomology class of $c_1(L)$, which we denote by $\mathcal{T}(X, \omega)$.
		The latter space is compact (in the weak topology of currents), cf. \cite[Proposition 8.5]{GuedjZeriahBook}.
		\par 
		The functional $F$ is continuous by Hartogs' lemma, cf. \cite[Theorem 1.46]{GuedjZeriahBook}.
		Note that by \cite[Lemma 1.14]{BerBoucNys}, the functional $F_1$ is also continuous if the measure $\mu$ is non-pluripolar.
		In this case the bound (\ref{eq_bnd_low}) is immediate, since continuous functions are bounded on compact sets.
		In the general setting, we shall adapt the proof of \cite[Lemma 1.14]{BerBoucNys} to obtain the desired lower bound.
		\par 
		Assume -- for the sake of contradiction -- that there exists a sequence of elements $u_i \in \mathcal{T}(X,\omega)$ such that $F_1(u_i) - F(u_i) \to -\infty$, or, equivalently, 
		\begin{equation}\label{eq_conv_exp_0}
			\int_X \exp(u_i) d\mu \to 0,
		\end{equation}
		where -- by normalizing -- we view $u_i$ as a sequence in ${\rm{PSH}}(X, \omega)$ satisfying $F(u_i)=0$.  
		Recall that Hartogs' lemma states that, up to passing to a subsequence, there exists $u \in {\rm PSH}(X,\omega)$ such that $u_i \to u$ in $L^1(X)$, and moreover 
		\begin{equation}\label{eq_hart_pp_e}
			\limsup_{i \to \infty} u_i = u \quad \text{outside of a pluripolar subset } E.
		\end{equation}
		By taking again a subsequence, we can further assume that $u_i$ converge to $u$ Lebesgue almost everywhere, as $i \to \infty$.
		\par 
		By repeating the functional-analytic argument of \cite[Lemma 1.14]{BerBoucNys}, based only on the Banach-Alaoglu theorem, we may, for each $i \in \nat$, find a convex combination
		\begin{equation}\label{eq_convex_combination}
    		v_i := \sum_{j \in I_i} t_{i,j} \cdot \exp(u_j),
		\end{equation}
		where $I_i \subset \{i, i+1, \ldots\}$ is a finite subset, such that $v_i \to v$ in $L^2(\mu)$ for some $v$.  
		Hence, upon passing to a subsequence, there exists a Borel subset $F \subset X$ with $\mu(F)=0$ 
such that $v_i \to v$ pointwise on $X \setminus F$.
		By (\ref{eq_conv_exp_0}) and (\ref{eq_convex_combination}), we then conclude that $v = 0$ on $X \setminus F$.
		\par 
		At the same time, since $u_i$ converge to $u$ Lebesgue almost everywhere, we see that $v_i$ converge to $\exp(u)$ Lebesgue almost everywhere.
		By the Hartogs' lemma, we conclude that outside of a pluripolar subset $H \subset X$, $\log(v_i)$ converges towards $u$ (this is applicable since standard properties of plurisubharmonic functions ensure that $\log(v_i) \in {\rm{PSH}}(X, \omega)$). 
		We note $G := \{ x \in X : u(x) = - \infty \}$.
		By (\ref{eq_hart_pp_e}) and the above remark, we deduce that $X \setminus F \subset G \cup E \cup H$.
		Note that $G \cup E \cup H$ is a pluripolar subset verifying $\mu(G \cup E \cup H) \geq \mu(X \setminus F) = \mu(X)$, contradicting our initial assumption that $\mu$ does not give full mass to pluripolar subsets.
	\end{proof}
	We now decompose the divisor $E$ into irreducible components $E = \sum_{i = 1}^{r} a_i E_i$, $a_i \in \nat$. 
	The next lemma is a version of a well-known tautological maximal principle for the $L^{\infty}$-norm.
	The only minor difference is that we apply it for sections which are vanishing along $E$.
	\par 
	For this, recall that for a psh function $\phi$, defined in a neighborhood of $a \in \comp^n$, the Lelong number $\nu_a(\phi)$ of $\phi$ at $a$ is defined as 
	\begin{multline}
		\nu_a(\phi) = \sup \Big\{ \gamma \geq 0 : \text{ there are $C_{\gamma} > 0$, and an open subset $U_{\gamma} \subset \comp^n$, $a \in U_{\gamma}$,} \\ \text{so that } \phi(z) \leq \gamma \cdot \log \|z - a\| + C_{\gamma} \text{ for any } z \in U_{\gamma} \Big\}.
	\end{multline}
	For an irreducible subvariety $Z \subset \comp^n$, we define the generic Lelong number of $\phi$ along $Z$ as 
	\begin{equation}
		\nu_Z(\phi) = \inf_{x \in Z} \nu_x(\phi).
	\end{equation}
	We now define the following envelope
	\begin{equation}\label{eq_env_lelong}
		\phi^{\epsilon, E} := \sup \Big\{ \phi \in {\textrm{PSH}}(X, \omega) : \phi \leq 0 \text{, and } \nu_{E_i}(\phi) \geq \epsilon a_i, \text{ for any $i = 1, \ldots, r$} \Big\}.
	\end{equation}
	\begin{lem}\label{lem_max_principl}
		The norms ${\textrm{Ban}}_k^{\infty}(h^L)$ and ${\textrm{Ban}}_k^{\infty}(h^L \cdot \exp(-\phi^{\epsilon, E}))$ coincide when restricted to the subspace $H^0(X, L^{\otimes k} \otimes \mathscr{O}_X(- \lceil \epsilon k E \rceil)) \subset H^0(X, L^{\otimes k})$.
	\end{lem}
	\begin{proof}
		Note first that since $\phi^{\epsilon, E} \leq 0$, we immediately get the inequality ${\textrm{Ban}}_k^{\infty}(h^L) \leq {\textrm{Ban}}_k^{\infty}(h^L \cdot \exp(-\phi^{\epsilon, E}))$.
		To prove Lemma \ref{lem_max_principl}, it remains to establish the reverse inequality.
		\par 
		For a given section $s \in H^0(X, L^{\otimes k} \otimes \mathscr{O}_X(- \lceil \epsilon k E \rceil)) \setminus \{ 0 \}$, we consider the function $\phi(x) := \frac{1}{k} \log |s(x)|_{(h^L)^k} -  \frac{1}{k} \log \| s \|_{{\textrm{Ban}}_k^{\infty}(h^L)}$.
		By the Poincaré-Lelong formula, $\phi$ is one of the admissible candidates in the supremum defining (\ref{eq_env_lelong}).
		The obvious inequality $\phi(x) \leq \phi^{\epsilon, E}(x)$ is simply a restatement of the pointwise bound $|s(x)|_{(h^L \cdot \exp(- \phi^{\epsilon, E}))^k} \leq \| s \|_{{\textrm{Ban}}_k^{\infty}(h^L)}$.
		Taking the supremum over $x \in X$ in the above inequality yields the desired conclusion.
	\end{proof}
	\par 
	For the last auxiliary result, we denote by $s_{E_i} \in H^0(X, \mathscr{O}_X(E_i))$ the canonical holomorphic section verifying ${\textrm{div}}(s_{E_i}) = E_i$.
	We endow each line bundle $\mathscr{O}_X(E_i)$, $i = 1, \ldots, r$, with a smooth metric $h^{E_i}$, and denote by $x \mapsto |s_{E_i}(x)|$ the norms of the associated sections.
	\begin{lem}\label{lem_sing_env}
		Assume that $\epsilon > 0$ is small enough so that $c_1(L) - \epsilon \cdot \sum_{i = 1}^{r} a_i \cdot c_1(\mathscr{O}_X(E_i))$ is a big class (i.e. containing a strictly positive $(1, 1)$-current).
		Then there is $C > 0$ so that for any $x \in X$,
		\begin{equation}
			\phi^{\epsilon, E}(x) \leq \epsilon \cdot \sum_{i = 1}^{r} a_i \cdot \log |s_{E_i}(x)| + C.
		\end{equation}
	\end{lem}
	\begin{proof}
		Recall that Berman in \cite[Lemma 4.1]{BermanEnvProj} established the following identity
		\begin{equation}\label{eq_berm_ident}
			\phi^{\epsilon, E} - \epsilon \cdot \sum_{i = 1}^{r} a_i \cdot \log |s_{E_i}|
			=
			\phi^{\epsilon, E}_{0}, 
		\end{equation}
		where $\phi^{\epsilon, E}_{0}$ admits for $\omega_{\epsilon} := \omega - \epsilon \cdot 2 \pi \cdot \sum_{i = 1}^{r} a_i \cdot c_1(\mathscr{O}_X(E_i), h^{E_i})$ the following description
		\begin{equation}
			\phi^{\epsilon, E}_{0} := \sup \Big\{ \phi \in {\textrm{PSH}}(X, \omega_{\epsilon}) : \phi \leq - \epsilon \cdot \sum_{i = 1}^{r} a_i \cdot \log |s_{E_i}| \Big\}.
		\end{equation}
		Note that in \cite[Lemma 4.1]{BermanEnvProj}, it was assumed that $E$ is irreducible, but the proof holds without this assumption. 
		Note that our assumption on $\epsilon > 0$ implies that ${\textrm{PSH}}(X, \omega_{\epsilon})$ is non-empty.
		\par 
		Now, let $K$ be an arbitrary non-pluripolar subset not intersecting $E$ (for example, a small ball in some local coordinates around $x \notin E$).
		Then for $C := \sup_{x \in K} ( - \epsilon \cdot \sum_{i = 1}^{r} a_i \cdot \log |s_{E_i}(x)| )$, we have $C < +\infty$. 
		Let us now define $\phi_K := \sup \{ \phi \in {\textrm{PSH}}(X, \omega_{\epsilon}) : \phi \leq C \text{ on } K \}$.
		The non-pluripolarity of $K$ and non-emptiness of ${\textrm{PSH}}(X, \omega_{\epsilon})$ assure that $\phi_K$ is uniformly bounded from above, cf. \cite[Theorem 9.17]{GuedZeriGeomAnal}.
		The bound of Lemma \ref{lem_sing_env} follows immediately from this, (\ref{eq_berm_ident}) and the obvious bound $\phi^{\epsilon, E}_{0} \leq \phi_K$.
	\end{proof}
	With all these preliminaries, we can finally complete our proof.
	\begin{proof}[Proof of Theorem \ref{thm_forbid_region}.]
		Note first that it suffices to establish the result for $\epsilon > 0$ small enough.
		From now on, we choose $\epsilon$ as in Lemma \ref{lem_sing_env}.
		Remark first the following characterization of the partial Bergman kernel
		\begin{equation}\label{eq_bergman_max}
			B_k^{\epsilon, Y}(x, x)
			=
			\sup \frac{|s(x)|_{(h^L)^k}^2}{\|s\|_{{\textrm{Hilb}}_k(h^L, \mu)}^2},
		\end{equation}
		where the supremum is taken over all $s \in H^0(X, L^{\otimes k} \otimes \mathscr{O}_X(- \lceil \epsilon k E \rceil)) \setminus \{0\}$.
		Now, according to Proposition \ref{lem_asympt_bm}, there is a continuous metric $h^L_0$ on $L$, so that for any $k \in \nat$, we have
		\begin{equation}\label{eq_bergman_max1}
			\|s\|_{{\textrm{Hilb}}_k(h^L, \mu)} 
			\geq
			\|s\|_{{\textrm{Ban}}_k^{\infty}(h^L_0)}.
		\end{equation}
		By Lemma \ref{lem_max_principl}, for any $s \in H^0(X, L^{\otimes k} \otimes \mathscr{O}_X(- \lceil \epsilon k E \rceil)) \setminus \{0\}$, we have
		\begin{equation}\label{eq_bergman_max2}
			\|s\|_{{\textrm{Ban}}_k^{\infty}(h^L_0)}
			=
			\|s\|_{{\textrm{Ban}}_k^{\infty}(h^L_0 \cdot \exp(-\phi^{\epsilon, E}_0))},
		\end{equation}
		where $\phi^{\epsilon, E}_0$ is defined analogously to (\ref{eq_env_lelong}) where we take a reference metric $h^L_0$ instead of $h^L$.
		Note, however, that by Lemma \ref{lem_sing_env}, there is a neighborhood $U$ of $Y$ so that 
		\begin{equation}\label{eq_bergman_max22}
			{\textrm{Ban}}_k^{\infty}(h^L_0 \cdot \exp(-\phi^{\epsilon, E}_0))
			\geq 
			{\textrm{Ban}}_k^{\infty}(U, \exp(\delta) \cdot h^L).
		\end{equation}
		By the definition of the $L^{\infty}$-norm, for any $x \in U$, we have
		\begin{equation}\label{eq_bergman_max3}
			\|s\|_{{\textrm{Ban}}_k^{\infty}(U, e \cdot h^L)}
			\geq
			\exp(k \delta) \cdot |s(x)|_{(h^L)^k}.
		\end{equation}
		Hence, from (\ref{eq_bergman_max})-(\ref{eq_bergman_max3}), we conclude.
	\end{proof}
	
	\subsection{Diagonal Bergman measures put no mass on closed pluripolar subsets}\label{sect_bergm_pp}
	The main result of this section refines Theorem \ref{thm_anal_no_mass} by showing that any subsequential limit of diagonal Bergman measures assigns no mass to pluripolar subsets.  
	Although this result is not needed in the subsequent sections, we have included it because the proof may be of independent interest, and it highlights a natural question discussed in the end of the section.
	\par 
	\begin{thm}\label{thm_pp_no_mass}
		For any Borel measure $\mu$ which does not give full mass to pluripolar subsets, any closed pluripolar subset $E \subset X$ and any $\epsilon > 0$, there exists an open neighborhood $U$ of $E$ and $k_0 \in \nat$, so that for any $k \geq k_0$, we have
		\begin{equation}\label{eq_pp_no_mass}
			\int_U \mu_k^{\rm{Berg}, \Delta} \leq \epsilon.
		\end{equation}
	\end{thm}
	From Theorem \ref{thm_pp_no_mass}, we obtain the following result.
	\begin{cor}\label{cor_nomalss}
		Any subsequential limit of $\mu_k^{\rm{Berg}, \Delta}$ doesn't place mass on pluripolar subsets.
	\end{cor}
	Before proceeding with the proof of Corollary \ref{cor_nomalss}, we note that Theorem \ref{thm_pp_no_mass} immediately implies that any subsequential limit of $\mu_k^{\rm{Berg}, \Delta}$ doesn't place mass on $F_{\sigma}$-pluripolar subsets, that is, to sets which can be expressed as a countable union of closed subsets.
	Note, however, that not every complete pluripolar subset is an $F_{\sigma}$, as illustrated by the following example kindly communicated to us by Franck Wielonsky.
	\par 
	\textbf{Example}.
	Let $C \subset \mathbb{C}$ be a polar Cantor set, and let $D \subset C$ be a dense countable subset. Consider the set $E := C \setminus D$. 
	We claim that $E$ is complete polar (i.e., there exists a subharmonic function $u$ such that $E = \{x : u(x) = -\infty\}$) and that it is not $F_\sigma$.  
	First, since $C$ is closed, it is also a $G_\delta$ set, i.e. a countable intersection of open sets. 
	Similarly, the complement of the countable set $D$, $\mathbb{C} \setminus D = \cap_{p \in D} (\mathbb{C} \setminus \{p\})$, is a $G_\delta$ set. Therefore, $E = C \cap (\mathbb{C} \setminus D)$ is also a $G_\delta$ set. 
	Moreover, $E$ is polar, since it is contained in the polar set $C$. By Deny's theorem \cite{DenyThmGDelta} (cf. \cite[\S 3.5]{Ransford}), it follows that $E$ is complete polar.  
	To see that $E$ is not $F_\sigma$, suppose for contradiction that it is. Then we can write $E = \bigcup_{n=1}^\infty F_n$, where $F_n$ are closed.
	Since $D$ is dense in $C$, each $F_n \subset E$ is nowhere dense in $C$. 
	Also since $C$ is a perfect set, every point of it is nowhere dense in it.
	Consequently, $C$ would be covered by the closed nowhere dense sets $F_n$ together with the points of $D$, giving a countable union of closed nowhere dense sets covering $C$. 
	This contradicts Baire's theorem, as $C$ is a complete metric space. Hence, $E$ cannot be $F_\sigma$.
	We also note that higher-dimensional analogues of the above phenomenon can be obtained by considering the sets $E \times \mathbb{C}^r \subset \mathbb{C}^{r+1}$.
	These sets are complete pluripolar, since $E$ is complete polar, but they are not $F_\sigma$, as the property of not being $F_\sigma$ is preserved under taking products with $\mathbb{C}^r$.
	\par 
	\begin{proof}[Proof of Corollary \ref{cor_nomalss}]
		Let us consider a subsequential limit of $\mu_k^{\rm{Berg}, \Delta}$, which we denote by $\nu$.
		Being a limit of Radon measures, $\nu$ is Radon.
		In particular, it is inner regular, see \cite[Definition 1.9 and Theorems 1.8, 1.39, 1.41]{EvansGariepy}.
		Note that any complete pluripolar subset $E$ is a $G_\delta$ set (in fact, if $E = \{ x \in X : \psi(x) = - \infty \}$, where $\psi$ is a (local) psh function, then $E = \cap_{n \in \nat} \{ x \in X : \psi(x) < -n \}$).
		In particular, $E$ is Borel.
		By inner regularity of $\nu$, we deduce 
		\begin{equation}\label{eq_innreg}
			\nu(E) = \sup \Big\{ \nu(K) : K \text{ compact}, K \subset E \Big\}.
		\end{equation}
		\par  
		Note, however, that any $K$ as above, being a subset of $E$, is also pluripolar. 
		Hence, we have $\nu(K) = 0$ by Theorem \ref{thm_pp_no_mass}. 
		But then by (\ref{eq_innreg}), we have $\nu(E) = 0$.
		This shows that any subsequential limit of $\mu_k^{\rm{Berg}, \Delta}$ doesn't place mass on complete pluripolar subsets.
		As an arbitrary pluripolar subset is a subset of a complete pluripolar subsets, it finishes the proof.
	\end{proof}
	
	Despite the similarity between the statements of Theorems \ref{thm_anal_no_mass} and \ref{thm_pp_no_mass}, their proofs differ drastically, as the analogue of (\ref{eq_vol_cont}) doesn't hold when $\mathcal{J}_{Y}$ is replaced by the sheaf of holomorphic sections vanishing along $E$ (for instance, consider the case when $E$ is an infinite set of points with a single accumulation point).
	The following statement remedies this problem.
	\begin{prop}\label{prop_ban_inf_2hl}
		For any $\epsilon, \delta > 0$, there is $k_0 \in \nat$, so that for any $k \geq k_0$, there is a vector subspace $E_k \subset H^0(X, L^{\otimes k})$, verifying 
		\begin{equation}\label{eq_dim_ek_big}
			\frac{\dim E_k}{\dim H^0(X, L^{\otimes k})} \geq (1 - \epsilon),
		\end{equation}
		and an open neighborhood $U$ of $E$ so that we have
		\begin{equation}\label{eq_ban_inf_2hl}
			{\textrm{Hilb}}_k(h^L, \mu)
			\geq 
			{\textrm{Ban}}_k^{\infty}(U, \exp(\delta) \cdot h^L), \quad \text{when restricted to } E_k.
		\end{equation}
	\end{prop}
	Once the above statement is established, the proof of Theorem \ref{thm_pp_no_mass} proceeds along the same lines as the proof of Theorem \ref{thm_anal_no_mass}. 
	More specifically, instead of the subspace $H^0(X, L^{\otimes k} \otimes \mathcal{J}_{Y}^{\lceil \epsilon' k \rceil})$ in Section \ref{sect_bergm_anal}, one uses $E_k$: the continuity of the volume function from (\ref{eq_vol_cont}) is then replaced by (\ref{eq_dim_ek_big}).
	The existence of the “forbidden region" for the partial Bergman kernel associated with the subspace $E_k$ is established as in Theorem \ref{thm_forbid_region}, where (\ref{eq_bergman_max2}) and (\ref{eq_bergman_max22}) are replaced by (\ref{eq_ban_inf_2hl}). Apart from this adjustment, the proof carries over verbatim, and we omit the details.
	\par 
	The proof of Proposition \ref{prop_ban_inf_2hl} will be based on the calculation of volumes of certain balls on the cohomology $H^0(X, L^{\otimes k})$.
	To explain this, we first introduce some preliminary notions.
	We fix a smooth Hermitian metric $h^L_0$ on $L$ and denote by $\omega := 2 \pi c_1(L, h^L_0)$.
	For any bounded $\phi : X \to \real$, we define the psh envelope as
	\begin{equation}\label{defn_env_sic}
		P(\phi)
		:=
		\sup \Big\{
			\phi_0 \in {\textrm{PSH}}(X, \omega) : \phi_0 \leq \phi
		\Big\}.
	\end{equation}
	For $h^L := h^L_0 \cdot \exp(- \phi)$, we denote $P(h^L) := h^L_0 \cdot \exp(-P(\phi))$, and let $P(h^L)_*$ be the lower-semicontinuous regularization of $P(h^L)$, given by $h^L_0 \cdot \exp(-P(\phi)^*)$ where $P(\phi)^*$ is the upper-semicontinous regularization of $\phi$.
	It is classical that $P(\phi)^* \in {\textrm{PSH}}(X, \omega)$.
	\par 
	We define the \textit{logarithmic relative spectrum} of a norm $N_1$ on a finitely dimensional vector space $V$, $\dim V =: r$, with respect to another norm $N_2$ on $V$, as a non-increasing sequence $\lambda_j := \lambda_j(N_1, N_2) \in \real$, $j = 1, \cdots, r$, defined so that
	\begin{equation}\label{eq_log_rel_spec}
		\lambda_j
		:=
		\sup_{\substack{W \subset V \\ \dim W = j}} 
		\inf_{w \in W \setminus \{0\}} \log \frac{\| w \|_2}{\| w \|_1}.
	\end{equation}
	For $p \in [1, + \infty[$, we then let
	\begin{equation}\label{eq_dp_defn_norms}
		d_p(N_1, N_2) := \sqrt[p]{\frac{\sum_{i = 1}^{r} |\lambda_i|^p}{r}}.
	\end{equation} 
	Note that one can alternatively define 
	\begin{equation}\label{eq_log_rel_spec0}
		\lambda_{j, 0}
		:=
		\inf_{\substack{W \subset V \\ {\textrm{codim}} W = j - 1}} 
		\sup_{w \in W \setminus \{0\}} \log \frac{\| w \|_2}{\| w \|_1},
	\end{equation}
	and $d_{p, 0}(N_1, N_2)$ as in (\ref{eq_dp_defn_norms}) with $\lambda_{j, 0}$ in place of $\lambda_j$.
	Remark that if both $N_1$ and $N_2$ are Hermitian, $\lambda_j$ and $\lambda_{j, 0}$ coincide.
	We claim that for arbitrary norms, we have 
	\begin{equation}\label{eq_dp_same}
		\big| 
		d_{p, 0}(N_1, N_2) - d_{p}(N_1, N_2)
		\big|
		\leq
		2 \log \dim V.
	\end{equation}
	\par
	Indeed, recall that John ellipsoid theorem, cf. \cite[\S 3]{PisierBook}, says that for any normed vector space $(V, N_V)$, there is a Hermitian norm $H_V$ on $V$, verifying 
	\begin{equation}\label{eq_john_ellips}
		H_V \leq N_V \leq \sqrt{\dim V} \cdot H_V.
	\end{equation}
	We then see that (\ref{eq_dp_same}) follows from this, Minkowski inequality and the fact that for Hermitian $N_1, N_2$, we have $d_{p, 0}(N_1, N_2) = d_{p}(N_1, N_2)$.
	\par 
	Following established tradition, we shall use $d_p(N_1, N_2)$ throughout our subsequent analysis; however, thanks to the bound (\ref{eq_dp_same}), all our results will also hold for $d_{p, 0}(N_1, N_2)$.
	The advantage of $d_{p, 0}(N_1, N_2)$ is that immediately from definitions, we see the following: assume $N_1$ and $N_2$ are such that $d_{p, 0}(N_1, N_2) \leq \epsilon \cdot c$ for some $\epsilon, c > 0$, then there is a vector subspace $E \subset V$ verifying 
	\begin{equation}\label{eq_dp_reform_subsp}
		\dim E \geq (1 - \epsilon) \dim V, \qquad N_1 \geq \exp(-2 c) \cdot N_2, \quad \text{when restricted to } E.
	\end{equation}
	Indeed, for $j := \lceil \epsilon \dim V \rceil$, we see that $\lambda_{j, 0} \leq c$.
	Hence, if one takes a subspace $E$ close to realizing infimum from (\ref{eq_log_rel_spec0}), it would verify both statements of (\ref{eq_dp_reform_subsp}).
	\par 
	\begin{sloppypar}
	We say that the graded norms $N = \oplus_{k = 0}^{\infty} N_k$ and $N' = \oplus_{k = 0}^{\infty} N_k'$ on $R(X, L) := \oplus_{k = 0}^{+ \infty} H^0(X, L^{\otimes l})$ are $p$-\textit{equivalent} ($N \sim_p N'$) if 
 	\begin{equation}\label{defn_equiv_relp}
		\frac{1}{k} d_p(N_k, N_k') \to 0, \qquad \text{as } k \to \infty.
	\end{equation}
	In \cite[\S 2.3]{FinNarSim}, we established that $\sim_p$, $p \in [1, +\infty[$, is an equivalence relation.
	\end{sloppypar}
	\par 
	The above equivalence relation can be reformulated in the language of volumes of balls associated with the fixed norms.
	To state this precisely, on a finitely-dimensional Hermitian vector space $(V, H)$, endowed with a norm $N_0$ on $V$, we denote by ${\rm{vol}} (N)$ the volume of the unit ball of $N$ (calculated with respect to the volume element associated with $H$).
	Note that while ${\rm{vol}} (N)$ clearly depends on the choice of $H$, if we fix another norm $N_1$ on $V$, the change of variables formula implies that the difference $\log {\rm{vol}} (N_0) - \log {\rm{vol}} (N_1)$ is independent of the choice of $H$.
	\begin{prop}[{\cite[Proposition 5.4]{FinEigToepl}}]\label{prop_d_p_and_vol}
		For an arbitrary sequence of norms $N = \oplus_{k = 0}^{\infty} N_k$ and $N' = \oplus_{k = 0}^{\infty} N_k'$ on $R(X, L)$, we have $a) \Rightarrow b)$ in the following statements
		\begin{enumerate}[a)]
			\item We have $N \sim_1 N'$.
			\item We have \begin{equation}
				\lim_{k \to \infty}
		\frac{\big|
		\log {\rm{vol}} (N_k) - \log {\rm{vol}} (N_k')
		\big|}{k \cdot \dim H^0(X, L^{\otimes k})}
		=
		0.
			\end{equation}
		\end{enumerate}
		If we have $N_k \geq N_k'$ for any $k \in \nat$, then we also have $b) \Rightarrow a)$.
		If there is $c > 0$, so that for any $k \in \nat^*$, we have $\exp(-c k) \cdot N_k' \leq N_k \leq \exp(c k) \cdot N_k'$, then the condition $a)$ above is equivalent to
		\begin{enumerate}[a)]
			\setcounter{enumi}{2}
			\item We have $N \sim_p N'$ for any $p \in [1, +\infty[$.
		\end{enumerate}
	\end{prop}
	\par 
	We fix now a smooth $(1, 1)$-form $\omega$.
	For any smooth function $u : X \to \real$, we define $\omega_u := \omega + \imun \partial \dbar u$.
	Recall that the \textit{Monge-Ampère energy functional} $\mathscr{E}$ on the space of smooth functions $u, v : X \to \real$ is defined (up to a universal constant) so that
	\begin{equation}\label{eq_energy}
		\mathscr{E}(u) - \mathscr{E}(v)
		=
		\frac{1}{(n + 1) \int_X \omega^n}
		\sum_{j = 0}^{n} \int_X (u - v) w_u^j \wedge w_v^{n - j}.
	\end{equation}
	\par 
	Using the definition of Bedford-Taylor for the powers $ w_u^j \wedge w_v^{n - j}$, one can extend the above definition to bounded $u, v \in {\rm{PSH}}(X, \omega)$.
	When the cohomological class $[\omega]$ of $\omega$ is merely big, there might be no bounded functions in ${\rm{PSH}}(X, \omega)$.
	Nevertheless, the above definition can still be extended if instead one uses the non-pluripolar products of currents due to Boucksom-Eyssidieux-Guedj-Zeriahi \cite{BEGZ}.
	Then $\mathscr{E}$ is monotone, i.e. for any $u \leq v$, we have $\mathscr{E}(u) \leq \mathscr{E}(v)$, cf. \cite[Proposition 10.14]{GuedjZeriahBook}.
	Moreover, the following theorem due to Berman-Boucksom \cite{BermanBouckBalls} holds.
	\begin{thm}\label{thm_balls}
		For any continuous metrics $h^L_0$, $h^L_1$ on $L$, the following identity holds
		\begin{equation}\label{eq_balls}
			\lim_{k \to \infty} \frac{\log {\rm{vol}} ({\rm{Ban}}_k(h^L_0)) - \log {\rm{vol}} ({\rm{Ban}}_k(h^L_1))}{k \cdot \dim H^0(X, L^{\otimes k})} 
			=
			\mathscr{E}(P(h^L_0)_*)
			-
			\mathscr{E}(P(h^L_1)_*).
		\end{equation}
	\end{thm}
	\par 
	\begin{proof}[Proof of Proposition \ref{prop_ban_inf_2hl}]
		Let $h^L_0$ be a Hermitian metric on $L$ given by Proposition \ref{lem_asympt_bm}.
		We fix $M > 1$, and consider the metric $h^L_M$ on $L$, defined as follows 
		\begin{equation}\label{eq_hlm_defn}
			h^L_M(x) = \begin{cases}
				h^L_0(x) \cdot M, \quad & x \in E,
				\\
				h^L_0(x), \quad & \text{otherwise}.
			\end{cases}
		\end{equation}
		We choose $M > 0$ so that $h^L_M \geq \exp(4 \delta + 1) \cdot h^L$ over $E$.
		Note that $h^L_M$ is upper-semicontinuous as $E$ is closed and $M > 1$.
		Hence, there is a decreasing sequence of continuous metrics $h^L_i$, $i \in \nat$, which converges (pointwise) towards $h^L_M$.
		\par 
		We shall now establish that for any $\epsilon > 0$, there are $i, k_0 \in \nat$, so that for any $k \geq k_0$, we have
		\begin{equation}\label{eq_vol_hlm}
			d_1 \big({\rm{Ban}}_k(h^L_i)), {\rm{Ban}}_k(h^L_0) \big)
			\leq
			\epsilon k / 2.
		\end{equation}
		Note that by the obvious monotonicity of the volumes and $h^L_M \geq h^L_0$, for any $i \in \nat$, we have
		\begin{equation}
			\log {\rm{vol}} ({\rm{Ban}}_k(h^L_i)) - \log {\rm{vol}} ({\rm{Ban}}_k(h^L_0)) \leq 0.
		\end{equation}
		Hence, by Proposition \ref{prop_d_p_and_vol} and Theorem \ref{thm_balls}, to prove (\ref{eq_vol_hlm}), it suffices to establish that 
		\begin{equation}\label{eq_cont_energy}
			\lim_{i \to \infty}
			\mathscr{E}(P(h^L_i)_*)
			=
			\mathscr{E}(P(h^L_0)_*).
		\end{equation}
		From the continuity properties of the energy functional, cf. \cite[Proposition 3.3]{BermanBouckBalls}, to show (\ref{eq_cont_energy}), it suffices to establish that $P(h^L_i)_*$ decreases towards $P(h^L_0)_*$ outside of a pluripolar subset.
		Note that the decreasing nature of the sequence is immediate since $P(\cdot)$ is monotone.
		Also, by our construction, $h^L_i$ decrease towards $h^L_0$ outside of a pluripolar subset $E$, and so the limit of $P(h^L_i)_*$ coincides with $P(h^L_0)_*$ outside of a pluripolar subset by  \cite[Proposition 2.2.3]{GuedjLuZeriahEnv}, establishing (\ref{eq_vol_hlm}).
		\par
		We now fix $i \in \nat$ big enough so that (\ref{eq_vol_hlm}) holds.
		By (\ref{eq_dp_reform_subsp}) and (\ref{eq_vol_hlm}), we conclude that for any $k \geq k_0$, there is a vector subset $E_k \subset H^0(X, L^{\otimes k})$, verifying (\ref{eq_dim_ek_big}), and so that 
		\begin{equation}\label{eq_bnd_rest_el_ban0_i}
			{\rm{Ban}}_k(h^L_0) \geq \exp(- k) \cdot {\rm{Ban}}_k(h^L_i), \quad \text{when restricted to } E_k.
		\end{equation}
		Now, by the continuity of $h^L_i$ and our choice of the constant $M$, we deduce that there is a neighborhood $U$ of $E$ so that ${\rm{Ban}}_k(h^L_i) \geq {\rm{Ban}}_k(U, \exp((\delta + 1) k) \cdot h^L)$.
		A combination of this, Proposition \ref{lem_asympt_bm} and (\ref{eq_bnd_rest_el_ban0_i}) immediately implies (\ref{eq_ban_inf_2hl}).
	\end{proof}
	\begin{rem}\label{rem_eq_hlm}
		As a byproduct of (\ref{eq_vol_hlm}), we see that $\oplus_{k = 0}^{\infty} {\rm{Ban}}_k(h^L_M) \sim_p \oplus_{k = 0}^{\infty} {\rm{Ban}}_k(h^L_0)$ for any $p \in [1, +\infty[$, $M > 0$. 
		In other words, the asymptotic class of the $\sup$-norm is not changed by the change of the metric along the closed pluripolar subset.
		Naturally, this does not remain true when the pluripolar subset is not closed -- for example if $E$ is a discrete dense set.
	\end{rem}
	We will now discuss the following partial converse to Proposition \ref{lem_asympt_bm}.
	\begin{prop}\label{lem_asympt_bm_pc}
		Assume that the measure $\mu$ has support in a closed pluripolar subset $E$.
		Then there is no continuous metric $h^L_0$ on $L$, so that for any $k \in \nat$, we have
		\begin{equation}\label{eq_asympt_bm_pc}
			{\textrm{Hilb}}_k(h^L, \mu)
			\geq
			{\textrm{Ban}}_k^{\infty}(h^L_0).
		\end{equation}
	\end{prop}
	\begin{proof}
		Assume for the sake of a contradiction that such $h^L_0$ exists.
		We now fix $M > 1$ big enough so that the metric $h^L_M$ defined in (\ref{eq_hlm_defn}) verifies $h^L_M \geq \exp(2) \cdot h^L$ over $E$.
		By Remark \ref{rem_eq_hlm} and (\ref{eq_dp_reform_subsp}), we see that for any $\epsilon > 0$, there is $k_0 \in \nat$, so that for any $k \geq k_0$, there is a non-empty vector subspace $E_k \subset H^0(X, L^{\otimes k})$, verifying 
		\begin{equation}\label{eq_dp_reform_subspaaa}
			{\textrm{Ban}}_k^{\infty}(h^L_0) \geq \exp(-k) \cdot {\textrm{Ban}}_k^{\infty}(h^L_M), \quad \text{when restricted to } E_k.
		\end{equation}
		Note, however, that ${\textrm{Ban}}_k^{\infty}(h^L_M) \geq \exp(2k) \cdot {\textrm{Ban}}_k^{\infty}(E, h^L)$ by our choice of $M$.
		A combination of (\ref{eq_asympt_bm_pc}), (\ref{eq_dp_reform_subspaaa}) and the obvious bound ${\textrm{Ban}}_k^{\infty}(E, h^L)
			\geq 
			{\textrm{Hilb}}_k(h^L, \mu)$ yields a contradiction.
	\end{proof}
	\par 
	We conclude this section by formulating an open question.
	To formulate it, we first note that as the space of probability measures is weakly compact, we can always choose a subsequence $k(i)$, $i \in \nat$, so that $\mu_{k(i)}^{\rm{Berg}, \Delta}$ converges weakly towards a probability measure $\nu$.
	As $\nu$ is necessarily non-pluripolar by Corollary \ref{cor_nomalss}, by a singular version of Yau's theorem \cite{YauThm} due to Guedj-Zeriahi \cite[Theorem 11.18]{GuedjZeriahBook}, there is a unique closed positive $(1, 1)$-current $T$ in $c_1(L)$ so that 
	\begin{equation}\label{eq_calabi}
		T^n = \nu \cdot \int_X c_1(L)^n.
	\end{equation}
	\par 
	Now, any Hermitian norm $H_k$ on $H^0(X, L^{\otimes k})$ yields the positive closed $(1, 1)$-forms $\omega_k(H_k) := c_1(L, h^L) + \frac{\imun \partial \dbar}{2 \pi k} \log B_k(x)$ lying in the class $c_1(L)$; here $B_k(x)$ is the diagonal Bergman kernel associated with $H_k$ and $h^L$. 
	This leads to the following natural question.
	\par 
	\textbf{Question}: does the sequence $\omega_k({\textrm{Hilb}}_{k(i)}(h^L, \mu))$ necessarily converge, as $i \to \infty$? If so, does the limit coincide with $T$?
	\par 
	Note that when $L$ is ample, the forms $\omega_k(H_k)$ can alternatively be defined through the Kodaira embeddings $\iota_k : X \to \mathbb{P}(H^0(X, L^{\otimes k})^*)$ as $\omega_k(H_k) = \frac{1}{k} \iota_k^* \omega(H_k)$, where $\omega(H_k)$ is the Fubini-Study form on $\mathbb{P}(H^0(X, L^{\otimes k})^*)$.
	In particular, if $\mu$ is a volume form, $L$ is ample and $h^L$ is smooth and positive, the answer to the above question is positive by the celebrated theorem of Tian \cite{TianBerg}. 
	When $\mu$ is merely Bernstein-Markov and $L$ is big, the answer is also positive by Theorem \ref{thm_diag}.
	\par 
	Note that the above limit appears in the study of zeros of random holomorphic sections (generalizing random polynomials), and if the answer to the above question is true, then these zeros would equidistribute towards $T$, see \cite[Lemma 3.1]{ShiffZeldRandom}, \cite[\S 5.3]{MaHol} or \cite{BloomLevRandPoly}.
	A closely related result on the zeros of orthogonal polynomials was established by Simon \cite[Theorem 1.5]{SimonWeakConv}, but it is not so clear to the author if the methods of \cite{SimonWeakConv} adapt to the study of zeros of random polynomials.

	\subsection{Localization relative to weight variation}\label{sect_loc_diag}
	The main goal of this section is to establish that variations in the weight of the measure affect the asymptotic value of the diagonal Bergman kernel at a given point only through the local change of the weight at that point, rather than globally, i.e. to establish Theorem \ref{thm_diag_local}.
	Throughout the section, we conserve the notation from Theorem \ref{thm_diag_local}.
	\par 
	For the following proposition, we shall denote by $B^i_k(x, y)$, $i = 1, 2$, $k \in \nat$, the Bergman kernels associated with $\mu_i$.
	The following simple observation is an immediate corollary of (\ref{eq_bk_id}) and (\ref{eq_peak_char1}): for any $x \in X$, we have
	\begin{equation}\label{eq_berg_ident}
		B^1_k(x, x)^2
		\leq
		B^2_k(x, x)
		\cdot
		\int_{y \in X} |B^1_k(x, y)|^2 d \mu_2(y).
	\end{equation}
	
	\begin{proof}[Proof of Theorem \ref{thm_diag_local}.]
		It suffices to show that for any positive continuous function $g: X \to \real$, we have $\int g(x) \cdot d \mu_{k, 1}^{\rm{Berg}, \Delta}(x) - \int g(x) \cdot d \mu_{k, 2}^{\rm{Berg}, \Delta}(x) \to 0$, as $k \to \infty$.
		Fix such a function $g$.
		We will prove that for every $\epsilon_0 > 0$, there exists $k_0 \in \nat$, such that for any $k \geq k_0$, we have
		\begin{equation}\label{eq_diag_loc_desir}
			\int g(x) \cdot d \mu_{k, 1}^{\rm{Berg}, \Delta}(x)
			\leq
			\epsilon_0 + 
			\int g(x) \cdot d \mu_{k, 2}^{\rm{Berg}, \Delta}(x).
		\end{equation}
		By symmetry, this inequality implies the desired convergence.
		\par 
		To establish (\ref{eq_diag_loc_desir}), for any $\epsilon > 0$, we shall partition $X$ into a finite union of (non-intersecting) Borel subsets $E_i$, $i = 1, \ldots, N(\epsilon)$, $N(\epsilon) \in \nat$, such that there are compact subsets $K_i$ and open subsets $U_i$, nested as $E_i \subset K_i \subset U_i$, and such that for every $i = 1, \ldots, N(\epsilon)$, there are $f_i \in \real$ and $g_i \geq 0$, so that for any $x \in U_i$, we have 
		\begin{equation}\label{eq_thm_diag_local1}
			|f(x) - f_i| \leq \epsilon, \qquad |g(x) - g_i| \leq \epsilon.
		\end{equation}
		\par 
		According to (\ref{eq_berg_ident}), we have
		\begin{equation}\label{eq_thm_diag_local2}
			\int_{x \in E_i} \frac{B^1_k(x, x)^2}{B^2_k(x, x)}  d \mu_1(x)
			\leq
			\int_{x \in E_i} \int_{y \in X} |B^1_k(x, y)|^2 d \mu_1(x) d \mu_2(y).
		\end{equation}
		By the Cauchy-Schwartz inequality, just as in (\ref{eq_int_bk_peak3}), we get
		\begin{equation}\label{eq_thm_diag_local3}
			\Big( \int_{x \in E_i} B^1_k(x, x)  d \mu_1(x) \Big)^2 \leq \Big( \int_{x \in E_i} B^2_k(x, x)  d \mu_1(x) \Big) \cdot \Big( \int_{x \in E_i} \frac{B^1_k(x, x)^2}{B^2_k(x, x)}  d \mu_1(x) \Big).
		\end{equation}
		By Theorem \ref{thm_off_diag}, the compactness of $K_i$ and $X \setminus U_i$, we obtain that there is $k(\epsilon) \in \nat$, so that for any $k \geq k(\epsilon)$, $i = 1, \ldots, N(\epsilon)$, we have
		\begin{equation}\label{eq_thm_diag_local4}
			\int_{x \in E_i} \int_{y \in X \setminus U_i} |B^1_k(x, y)|^2 d \mu_1(x) d \mu_2(y)
			\leq
			\frac{\epsilon^2 \cdot n_k \cdot \exp(f_i)}{N(\epsilon)}.
		\end{equation}
		By decomposing the integral on the right-hand side of (\ref{eq_thm_diag_local2}) into two parts: over $E_i \times U_i$ and over $E_i \times X \setminus U_i$, and using the bounds (\ref{eq_thm_diag_local1}) with (\ref{eq_thm_diag_local4}), we obtain that for any $i = 1, \ldots, N(\epsilon)$, $k \geq k(\epsilon)$, the following bound holds
		\begin{multline}\label{eq_thm_diag_local6}
			\int_{x \in E_i} \int_{y \in X} |B^1_k(x, y)|^2 d \mu_1(x) d \mu_2(y)
			\\
			\leq
			\exp(f_i + \epsilon)
			\cdot
			\int_{x \in E_i} \int_{y \in U_i} |B^1_k(x, y)|^2 d \mu_1(x) d \mu_1(y)
			+
			\frac{\epsilon^2 \cdot n_k \cdot \exp(f_i)}{N(\epsilon)}.
		\end{multline}
		Moreover, by enlarging the integral from $E_i \times U_i$ to $E_i \times X$, by (\ref{eq_push_forw}), we obtain
		\begin{equation}\label{eq_thm_diag_local7}
			\int_{x \in E_i} \int_{y \in U_i} |B^1_k(x, y)|^2 d \mu_1(x) d \mu_1(y)
			\leq
			\int_{x \in E_i} B^1_k(x, x) d \mu_1(x).
		\end{equation}
		\par 
		For any $k \in \nat$, we denote by $I_k \subset \{1, 2, \ldots, N(\epsilon) \}$, the set of indices so that for any $i \in I_k$,
		\begin{equation}\label{eq_thm_diag_local5}
			\int_{x \in E_i} B^1_k(x, x)  d \mu_1(x)
			\leq
			\frac{\epsilon \cdot n_k}{N(\epsilon)},
		\end{equation}
		and for any $i \notin I_k$, the inverse inequality is satisfied.
		From (\ref{eq_thm_diag_local6}) and (\ref{eq_thm_diag_local7}), we see that for any $k \geq k(\epsilon)$, $i \notin I_k$, we have
		\begin{equation}\label{eq_thm_diag_local77}
			\int_{x \in E_i} \int_{y \in X} |B^1_k(x, y)|^2 d \mu_1(x) d \mu_2(y)
			\\
			\leq
			\exp(f_i + 2 \epsilon) \cdot \int_{x \in E_i} B^1_k(x, x) d \mu_1(x).
		\end{equation}
		Note also that immediately from (\ref{eq_thm_diag_local1}), for any $k \in \nat$, $i = 1, \ldots, N_k$, we have
		\begin{equation}\label{eq_thm_diag_local8}
			\int_{x \in E_i} B^2_k(x, x)  d \mu_1(x)
			\leq
			\exp(- f_i + \epsilon)
			\cdot
			\int_{x \in E_i} B^2_k(x, x)  d \mu_2(x).
		\end{equation}
		By (\ref{eq_thm_diag_local2}), (\ref{eq_thm_diag_local3}), (\ref{eq_thm_diag_local77}) and (\ref{eq_thm_diag_local8}), we deduce that for any $k \geq k(\epsilon)$, $i \notin I_k$, we have
		\begin{equation}\label{eq_thm_diag_local9}
			\int_{x \in E_i} B^1_k(x, x)  d \mu_1(x) \leq \exp(3 \epsilon) \cdot \int_{x \in E_i} B^2_k(x, x)  d \mu_2(x).
		\end{equation}
		\par 
		We are now finally ready to establish the bound (\ref{eq_diag_loc_desir}).
		By (\ref{eq_thm_diag_local1}) and (\ref{eq_thm_diag_local5}), for any $k \in \nat$,
		\begin{equation}
			\int g(x) \cdot B^1_k(x, x)  d \mu_1(x)
			\leq
			\epsilon \cdot n_k \cdot (1 + \| g \|_{L^{\infty}(X)})
			+
			\sum_{i \notin I_k}
			g_i \cdot
			\int_{x \in E_i} B^1_k(x, x)  d \mu_1(x).
		\end{equation}
		By this and (\ref{eq_thm_diag_local9}), for any $k \geq k(\epsilon)$, we deduce that 
		\begin{multline}\label{eq_thm_diag_local10}
			\int g(x) \cdot B^1_k(x, x)  d \mu_1(x)
			\\
			\leq
			\epsilon \cdot n_k \cdot (1 + \| g \|_{L^{\infty}(X)})
			+
			\exp(3 \epsilon)
			\cdot
			\sum_{i \notin I_k}
			g_i \cdot
			\int_{x \in E_i} B^2_k(x, x)  d \mu_2(x).
		\end{multline}
		Note, however, that by the positivity of $g$ and (\ref{eq_thm_diag_local1}), we conclude
		\begin{equation}\label{eq_thm_diag_local11}
			\sum_{i \notin I_k}
			g_i \cdot
			\int_{x \in E_i} B^2_k(x, x)  d \mu_2(x)
			\leq
			\epsilon \cdot n_k
			+
			\int g(x) \cdot B^2_k(x, x)  d \mu_2(x).
		\end{equation}
		Since $\epsilon$ can be chosen as small as we wish, a combination of (\ref{eq_thm_diag_local10}) and (\ref{eq_thm_diag_local11}) yields  (\ref{eq_diag_loc_desir}).
	\end{proof}

	\section{Applications to the theory of Toeplitz operators and singular spaces}\label{sect_appl}
	The purpose of this section is to present several applications of the localization principle, with a primary focus on Toeplitz operator theory.
	More precisely, this section is organized as follows.
	In Section \ref{sect_mat}, we discuss connections with the theory of orthogonal polynomials and Toeplitz matrices.
Sections \ref{sect_toepl} and \ref{sect_toepl_eq} are devoted to the proofs of Theorems \ref{cor_alg} and \ref{thm_distr}, respectively.
Finally, the last section examines the extension of our results to the setting of singular spaces.

	\subsection{Orthogonal polynomials and Toeplitz matrices}\label{sect_mat}
	The main goal of this section is to illustrate how Theorems \ref{cor_alg} and \ref{thm_distr} specialize to the classical setting of Toeplitz matrices and orthogonal polynomials.
	\par 
	Below, we place ourselves in the setting of $X = \mathbb{P}^1$, $L = \mathscr{O}(1)$, $\mu$ of compact support $K$ in $\comp \subset \mathbb{P}^1$, and $h^L$ as described after (\ref{eq_cd_orth}).
	Then the Bernstein-Markov condition on $\mu$  can be rewritten in the following way: for any $\epsilon > 0$, there is $C > 0$ so that for any polynomial $P$, 
	\begin{equation}\label{eq_bm_clas0}
		\sup_{x \in K} |P(x)|^2
		\leq
		C
		\cdot
		\exp(\epsilon \deg(P))
		\cdot
		\int |P(x)|^2 d \mu(x).
	\end{equation}
	The psh envelope $h^L_K$ can be expressed as $h^L_K = h^L \cdot \exp(-\phi_K)$, where the function $\phi_K$ is given by
	\begin{equation}
		\phi_K(z) := \sup \{ v(z) : v \in \mathcal{L}, v \leq 0 \text{ on } K \},
	\end{equation}
	and $\mathcal{L}$ is the Lelong class of functions, consisting of psh functions on $\comp^n$ so that up to a constant, they are bounded above by $\frac{1}{2} \log (1 + |z|^2)$, $z \in \comp^n$, cf. \cite[Example 1.2]{GuedZeriGeomAnal}.
	\par 
	For the first example, we let $K$ to be the unit circle $\mathbb{S}^1 \subset \comp$, and let $\mu$ be the Lebesgue measure on it.
	Remark that $\mathbb{S}^1$ is non-pluripolar, see \cite{SadullaevReal}, cf. \cite[Exercise 4.39.5]{GuedjZeriahBook}, and the measure $\mu$ is Bernstein-Markov with respect to $h^L$ by \cite[Corollary 1.8]{BerBoucNys}, \cite[Proposition 3.6]{BernsteinMarkovSurvey}.
	\par 
	\begin{sloppypar}
	For any function $f \in \ccal^0(\mathbb{S}^1, \real)$, $f \neq 0$, we consider the Fourier expansion $f(\theta) = \sum_{i = - \infty}^{+ \infty} a_j \exp(\imun j \theta)$, where $\theta \in [0, 2 \pi[$, gives a standard parametrization of $\mathbb{S}^1$. 
	Then $a_i = \overline{a}_{-i}$ and not all $a_i$ vanish.
	Using the identification described after (\ref{eq_cd_orth}), we see that the monomials provide an orthogonal basis of $(H^0(X, L^{\otimes k}), {\textrm{Hilb}}_k(h^L, \mu))$, and the operator $T_k(f)$ writes in this basis as the following \textit{Toeplitz matrix}
	\begin{equation}
		T_k[f]
		:=
		\begin{bmatrix}
		a_0 & a_{-1} & a_{-2} & \cdots & a_{-k} \\
		a_1 & a_0 & a_{-1} & \cdots & a_{-k+1} \\
		a_2 & a_1 & a_0 & \cdots & a_{-k+2} \\
		\vdots & \vdots & \vdots & \ddots & \vdots \\
		a_{k} & a_{k-1} & a_{k-2} & \cdots & a_0
		\end{bmatrix}.
	\end{equation}
	\end{sloppypar}
	\par 
	Note that due to $\mathbb{S}^1$-symmetry, the equilibrium measure associated with $(\mathbb{S}^1, h^L)$ is just the Lebesgue measure on the circle, cf. \cite[Exercise 9.8]{GuedjZeriahBook}.
	Immediately from this and Theorem \ref{thm_distr}, we recover the following classical statement.
	\begin{thm}[{Szegő first limit theorem \cite{SzegoFST}}]
		For any continuous $f : \mathbb{S}^1 \to \real$, $g: \real \to \real$,
		\begin{equation}
			\lim_{k \to \infty} \frac{1}{k + 1} \sum_{\lambda \in {\rm{Spec}} (T_k[f])} g(\lambda)
			=
			\frac{1}{2 \pi}
			\int_0^{2 \pi} g(f(\theta)) d \theta.
		\end{equation}
	\end{thm}
	\begin{rem}
		The statement holds under much laxer assumptions on $f$, see \cite{NikolskiBook} for details.
	\end{rem}
	\par 
	Theorem \ref{cor_alg}, on its turn, recovers the corresponding result from Grenander-Szeg{\"o} \cite[\S 7, 8]{GrenanSzego}.
	\par 
	For another special case, we consider a subset $K := [-1, 1] \subset \real \subset \comp$ and let $\mu$ be the Lebesgue measure on $K$.
	By the same reasons as before, the set $K$ is non-pluripolar, and the measure $\mu$ is Bernstein-Markov with respect to $h^L$.
	The classical calculation due to Lundin \cite{Lundin}, cf. \cite[p. 707]{BedfordTaylorEquil}, shows that the equilibrium measure is given by
	\begin{equation}
		\mu_{\mathrm{eq}}(K, h^L) = \frac{dx|_{[-1, 1]}}{\pi \cdot \sqrt{1 - x^2}}.
	\end{equation}
	\par 
	\begin{sloppypar}
	Now, let us denote by $L_n(z)$, $z \in \comp$, $n \in \nat$, the normalized Legendre polynomials.
	Recall that this means that $L_n(z)$ has degree $n$, and the following orthogonality relation holds
	\begin{equation}\label{eq_orth_poly}
		\int_{-1}^{1} L_n(x) L_m(x) dx = 1 \cdot \delta_{n m},
	\end{equation}
	where $\delta_{n m}$ is the Kronecker symbol.
	Using the identification described after (\ref{eq_cd_orth}), for any $k \in \nat$, $\{ L_n(z), n = 0, 1, \ldots, k \}$ provides an orthonormal basis of $(H^0(X, L^{\otimes k}), {\textrm{Hilb}}_k(h^L, \mu))$.
	\end{sloppypar}
	\par 
	Let us now consider a function $f \in \ccal^0([-1, 1], \real)$ and the associated Toeplitz operators $T_k(f)$.
	In the above basis the operator $T_k(f)$ writes as the following \textit{Toeplitz form} 
	\begin{equation}
		T_k\{ f \}
		:=
		\begin{bmatrix}
		a_{0 0} & a_{0 1} & \cdots & a_{0 k} \\
		a_{1 0} & a_{1 1} & \cdots & a_{1 k} \\
		\vdots & \vdots & \ddots & \vdots \\
		a_{k 0} & a_{k 1} & \cdots & a_{k k}
		\end{bmatrix},
	\end{equation}
	where $a_{i j} := \int_{-1}^{1} f(x) \cdot L_i(x) L_j(x) dx$, $i, j \in \nat$.
	Immediately from this and Theorem \ref{thm_distr}, we recover the following result.
	\begin{thm}[{Grenander-Szeg{\"o} \cite[p. 116]{GrenanSzego} and Nevai \cite[\S 5]{Nevai}}]
		For any continuous $f : [-1, 1] \to \real$, $g: \real \to \real$, we have
		\begin{equation}
			\lim_{k \to \infty} \frac{1}{k + 1} \sum_{\lambda \in {\rm{Spec}} (T_k\{ f \})} g(\lambda)
			=
			\frac{1}{\pi}
			\int_{-1}^{1} \frac{g(f(x)) dx}{\sqrt{1 - x^2}}.
		\end{equation}
	\end{thm}
	\par 
	Theorem \ref{cor_alg}, on its turn, recovers some of the results from Grenander-Szeg{\"o} \cite[\S 8.1]{GrenanSzego}.
	\par 
	To conclude, we recall that for symmetric convex subsets $K \subset \mathbb{R}^n$, an explicit formula for the equilibrium measure associated with $(K, h^L)$ was obtained by Bedford-Taylor \cite{BedfordTaylorEquil}. 
	According to Lundin's formulas for $h^L_K$ from \cite{Lundin}, see also \cite[p. 707]{BedfordTaylorEquil}, the pair $(K, h^L)$ is pluriregular. 
	It then follows from \cite[Proposition 1.13 and Theorem 1.14]{BerBoucNys} that the corresponding equilibrium measures satisfy the Bernstein-Markov property.  
	Thus, the analogues of the above result can be formulated for arbitrary symmetric convex subsets $K \subset \mathbb{R}^n$.

	\subsection{Algebraic and spectral aspects of Toeplitz operators}\label{sect_toepl}
	The main goal of this section is to apply the results of Section \ref{sect_local_main} to establish Theorems \ref{cor_alg} and \ref{thm_distr}.
	We shall do so in a more general framework, and for this, by a slight abuse of notations, we introduce the following definition.
	\begin{defn}\label{defn_toepl_sch}
		A sequence of operators $T_k \in {\enmr{H^0(X, L^{\otimes k})}}$, $k \in \nat$, is called a \textit{Toeplitz operator} if there is $C > 0$, such that for any $k \in \nat$, $\| T_k \| \leq C$, where $\| \cdot \|$ is the operator norm, and there is $f \in \ccal^0(K, \real)$, called the \textit{symbol} of $\{T_k\}_{k = 0}^{+ \infty}$, so that for any $\epsilon > 0$, $p \in [1, +\infty[$, there is $k_0 \in \nat$, such that for every $k \geq k_0$, we have
		\begin{equation}\label{eq_toepl_schatten}
			\big \| T_k -  T_k(f) \big \|_p \leq \epsilon,
		\end{equation}
		where $\| \cdot \|_p$ is the $p$-Schatten norm, defined for an operator $A \in {\enmr{V}}$, of a finitely-dimensional Hermitian vector space $(V, H)$ as $\| A \|_p = (\frac{1}{\dim V} {\rm{Tr}}[|A|^p])^{\frac{1}{p}}$, $|A| := (A A^*)^{\frac{1}{2}}$.
	\end{defn}
	\begin{rem}
		a) The standard definition of Toeplitz operators uses the operator norm rather than the $p$-Schatten norm in (\ref{eq_toepl_schatten}), and the measure $\mu$ is taken to be a volume form, cf. \cite{MaHol}. 
		The variant from Definition \ref{defn_toepl_sch} is a mix of two definitions from \cite{FinSubmToepl}, \cite{FinEigToepl}; see also Grenander-Szeg{\"o} \cite[\S 7.4]{GrenanSzego} for some earlier variants in the context of orthogonal polynomials.
		\par 
		b) 
		Standard estimates, cf. (\ref{eq_schatted_bnds}), imply that it suffices to verify (\ref{eq_toepl_schatten}) for a single $p \in [1, +\infty[$.
	\end{rem}
	Before we proceed, we add a word of caution.
	While in the classical setting when $\mu$ is a volume form, $L$ is ample and $h^L$ is smooth and positive, it is known that Toeplitz operators have unique symbols, in our setting, the situation is very different, as Proposition \ref{prop_symb_well_def} below shows.
	Nevertheless, the following result clearly generalizes Theorem \ref{cor_alg}.
	\begin{thm}\label{cor_prod_toepl}
		Consider two sequences of operators $T_k, T'_k \in {\enmr{H^0(X, L^{\otimes k})}}$, $k \in \nat$, forming Toeplitz operators with symbols $f \in \ccal^0(K, \real)$ and $g \in \ccal^0(K, \real)$.
		Then the sequence of operators $T_k \circ T'_k \in {\enmr{H^0(X, L^{\otimes k})}}$, $k \in \nat$, forms a Toeplitz operator with a symbol $f \cdot g \in \ccal^0(K, \real)$.
	\end{thm}
	We shall use the following result concerning the spectral radius.
	\begin{sloppypar}
	\begin{lem}\label{lem_spec_radius}
		For any $f \in \ccal^0(K, \real)$, the minimal and the maximal eigenvalues $\lambda_{\min}(T_k(f))$, $\lambda_{\max}(T_k(f))$ of $T_k(f)$ satisfy $\lambda_{\min}(T_k(f)) \geq \inf_{x \in K} f(x)$ and $\lambda_{\max}(T_k(f)) \leq \sup_{x \in K} f(x)$.
	\end{lem}
	\begin{proof}
		Replacing $f$ with $-f$, the problem reduces to analyzing $\lambda_{\max}(T_k(f))$.
		The bound then follows immediately from the min-max characterization of the eigenvalues and the bound $\langle T_k(f) s, s \rangle_{{\textrm{Hilb}}_k(h^L, \mu)} \leq \sup_{x \in K} f(x) \cdot \langle s, s \rangle_{{\textrm{Hilb}}_k(h^L, \mu)}$ for any $s \in H^0(X, L^{\otimes k})$, which follows immediately from the definition of the $L^2$-norm.
	\end{proof}
	\end{sloppypar}
	\begin{proof}[Proof of Theorem \ref{cor_alg}]
		We remark that the uniform bound on the operator norm of $T_k(f) \circ T_k(g)$ is an immediate consequence of Lemma \ref{lem_spec_radius}.
		Classical properties of Schatten norms yield -- in the notations of Definition \ref{defn_toepl_sch} -- that for any $T, S \in {\enmr{H^0(X, L^{\otimes k})}}$, $k \in \nat$, we have
		\begin{equation}\label{eq_schatted_bnds}
			\| T \|_p \leq \| T \|_2^{\frac{1}{p}} \cdot  \| T \|^{\frac{p - 1}{p}}, \qquad \| S \circ T \|_p \leq \| S \| \cdot \| T \|_p.
		\end{equation}
		So it suffices to establish that 
		\begin{equation}\label{eq_sk_schw00}
			\lim_{k \to \infty}
			\| T_k(f) \circ T_k(g) - T_k(f \cdot g) \|_2 
			=
			0.
		\end{equation}
		Now, it is immediate to see using the reproducing property $B_k \circ B_k = B_k$ that we can write
		\begin{equation}\label{eq_sk_schw0}
			T_k(f) \circ T_k(g) - T_k(f \cdot g) 
			=
			B_k \circ S_k \circ B_k,
		\end{equation}
		where $B_k$ is the Bergman kernel as in Definition \ref{defn_toepl_sch}, and $S_k \in {\enmr{H^0(X, L^{\otimes k})}}$, is defined as 
		\begin{equation}
			(S_k s)(x) := \int_{y \in X} S_k(x, y) \cdot s(y) \cdot d \mu(y), \quad \text{for any } s \in H^0(X, L^{\otimes k}),
		\end{equation}
		for $S_k(x, y) \in L^{\otimes k}_x \otimes (L^{\otimes k}_y)^*$ given by
		\begin{equation}\label{eq_sk_schw}
			S_k(x, y) = (f(x) g(x) - f(x) g(y)) B_k(x, y).
		\end{equation}
		From (\ref{eq_schatted_bnds}), we obtain
		\begin{equation}\label{eq_sk_schw1}
			\| B_k \circ S_k \circ B_k \|_2 \leq \| S_k \|_2.
		\end{equation}
		We now rely on the fact that the $2$-Schatten norm is the rescaled Hilbert-Schmidt norm, and the latter can be calculated using the $L^2$-norm of Schwartz kernel of the operator, which gives us 
		\begin{equation}\label{eq_sk_schw2}
			\| S_k \|_2^2
			=
			\frac{1}{n_k} \int_{X \times X} |S_k(x, y)|^{2}_{(h^L)^k} \cdot d \mu(x) \cdot d \mu(y).
		\end{equation}
		But from Theorem \ref{thm_off_diag} and (\ref{eq_sk_schw}), we deduce that 
		\begin{equation}
			\lim_{k \to \infty} \frac{1}{n_k} \int_{X \times X} |S_k(x, y)|^{2}_{(h^L)^k} \cdot d \mu(x) \cdot d \mu(y)
			=
			0,
		\end{equation}
		which along with (\ref{eq_sk_schw0}), (\ref{eq_sk_schw1}) and (\ref{eq_sk_schw2}) finishes the proof of (\ref{eq_sk_schw00}).
	\end{proof}	
	
	\begin{proof}[Proof of Theorem \ref{cor_prod_toepl}]
		The bound $\| T_k \circ T'_k \| \leq \| T_k \| \cdot \| T'_k \|$ implies that $\| T_k \circ T'_k \|$ is uniformly bounded.
		By the triangle inequality, we have
		\begin{multline}
			\big\| T_k \circ T'_k - T_k(f \cdot g) \big\|_p
			\leq
			\big\| (T_k - T_k(f)) \circ T'_k  \big\|_p
			\\
			+
			\big\| T_k(f) \circ (T'_k - T_k(g)) \big\|_p
			+
			\big\| T_k(f) \circ T_k(g) - T_k(f \cdot g) \big\|_p,
		\end{multline}
		which finishes the proof of Theorem \ref{cor_prod_toepl} by Theorem \ref{cor_alg} and (\ref{eq_schatted_bnds}).
	\end{proof}	
	
	Let us now derive the following useful statement concerning continuous functional calculus on Toeplitz operators.
	\begin{cor}\label{cor_func_calc}
		For any $h \in \ccal^0(\real)$ and a sequence of self-adjoint operators $\{ T_k , k \in \nat \}$, forming a Toeplitz operator with symbol $f \in \ccal^0(K, \real)$, $\{ h(T_k), k \in \nat \}$, forms a Toeplitz operator with symbol $h(f)$.
	\end{cor}
	\begin{proof}
		The uniform bound on the operator norm of $h(T_k)$ is immediate.
		Now, since $T_k$ has a uniformly bounded spectrum, let's say located inside of a compact interval $I$, the operator $h(T_k)$ depends solely on the restriction of $h$ on $I$.
		By Weierstrass approximation theorem, applied over $I$, we see that it suffices to establish the result for $h$ given by monomials.
		But this follows immediately from Theorem \ref{cor_prod_toepl} and finishes the proof.
		Note also that from Lemma \ref{lem_spec_radius} and the above proof, we see that the analogue of Corollary \ref{cor_func_calc} holds if we replace $T_k$ by $T_k(f)$ and $h$ by a continuous function on $[\min f, \max f]$.
	\end{proof}
	\par 
	From now on, we shall concentrate on the proof of Theorem \ref{thm_distr}.
	In particular, we shall assume that $\mu$ is Bernstein-Markov.
	We will need the following lemma, which -- although a special case of Theorem \ref{thm_distr} -- is actually used in its proof.
	\begin{lem}\label{lem_trace_toepl}
		For any $f \in \ccal^0(K, \real)$, we have $\lim_{k \to \infty} \frac{1}{n_k} {\rm{Tr}}[T_k(f)] = \int f \cdot \mu_{\mathrm{eq}}(K, h^L)$.
	\end{lem}
	\begin{proof}
		Immediately from the fact that trace can be calculated through the integral of the Schwartz kernel, we deduce
		\begin{equation}
			{\rm{Tr}}[T_k(f)]
			=
			\int_{X \times X} f(x) \cdot |B_k(x, y)|^{2}_{(h^L)^k} \cdot d \mu(x) \cdot d \mu(y).
		\end{equation}
		The result now follows from Theorem \ref{thm_diag} and (\ref{eq_push_forw}).
	\end{proof}
	We can now finally comment on the uniqueness of the symbol of a Toeplitz operator.
	For this, we introduce the following subset
	\begin{equation}\label{eq_kprim}
		K' := {\rm{supp}}(\mu_{\mathrm{eq}}(K, h^L)) \subset K.
	\end{equation}

	\begin{prop}\label{prop_symb_well_def}
		Let $f, g \in \ccal^0(K, \real)$.
		Then $\{ T_k(f) , k \in \nat \}$, forms a Toeplitz operator with symbol $g$ if and only if $f = g$ over $K'$.
	\end{prop}
	\begin{proof}
		According to Corollary \ref{cor_func_calc}, the sequence of operators $\{ |T_k(f) - T_k(g)| , k \in \nat \}$, forms a Toeplitz operator with symbol $|f - g|$.
		By Lemma \ref{lem_trace_toepl}, we conclude that
		\begin{equation}\label{eq_symb_wd_2}
			\lim_{k \to \infty} \frac{1}{n_k} {\rm{Tr}}[|T_k(f) - T_k(g)|] = \int |f - g| \cdot \mu_{\mathrm{eq}}(K, h^L).
		\end{equation}
		According to (\ref{eq_schatted_bnds}), $\{ T_k(f) , k \in \nat \}$, forms a Toeplitz operator with symbol $g$ if and only if
		\begin{equation}\label{eq_symb_wd_1}
			\lim_{k \to \infty} \frac{1}{n_k} {\rm{Tr}} \Big[ |T_k(f) - T_k(g)| \Big] = 0.
		\end{equation}
		Indeed, the quantity under the limit sign in (\ref{eq_symb_wd_1}) corresponds precisely to $\| T_k(f) - T_k(g) \|_1$, and by the uniform boundness of $T_k(f) - T_k(g)$, (\ref{eq_symb_wd_1}) implies that $\lim_{k \to \infty}  \| T_k(f) - T_k(g) \|_p = 0$ for any $p \in [1, +\infty[$.
		A comparison between (\ref{eq_symb_wd_2}) and (\ref{eq_symb_wd_1}) shows -- by the continuity of $f$ and $g$ -- that $f = g$ on $K'$ if and only if $\{ T_k(f) , k \in \nat \}$, forms a Toeplitz operator with symbol $g$.
	\end{proof}
	\begin{proof}[Proof of Theorem \ref{thm_distr}]
		Remark first that 
		\begin{equation}
			\sum_{\lambda \in {\rm{Spec}} (T_k(f))} g(\lambda)
			=
			{\rm{Tr}}[g(T_k(f))].
		\end{equation}
		The result now follows immediately from Corollary \ref{cor_func_calc} and Lemma \ref{lem_trace_toepl}.
		Note also that from Lemma \ref{lem_spec_radius}, the analogue of Theorem \ref{thm_distr} holds if we replace $T_k$ by $T_k(f)$ and $h$ by a continuous function on $[\min f, \max f]$.
	\end{proof}
	
	We finish this section with the following proposition concerning the spectral radius of Toeplitz operators.
	\begin{prop}\label{prop_spec_radius}
		In the notations of Lemma \ref{lem_spec_radius}, we have
		\begin{equation}
			\liminf_{k \to \infty} \lambda_{\max}(T_k(f)) \geq \max_{x \in K'} f(x),
			\qquad
			\limsup_{k \to \infty} \lambda_{\min}(T_k(f)) \leq \min_{x \in K'} f(x).
		\end{equation}
	\end{prop}
	\begin{sloppypar}
	\begin{proof}
		Replacing $f$ with $-f$, the problem reduces to analyzing $\lambda_{\max}(T_k(f))$.
		We let $r := \max_{x \in K'} f(x)$, and for a given $\epsilon > 0$, we consider a continuous function $g: \real \to [0, 1]$, so that $g(r) = 1$ and $g(y) = 0$ for $y < r - \epsilon$.
		Then by our assumption we have $\int_X g(f(x)) d \mu_{\mathrm{eq}}(K, h^L)(x) > 0$. 
		Proposition \ref{prop_spec_radius} follows from this and Theorem \ref{thm_distr}.
	\end{proof}
	\end{sloppypar}
	\par 
	We underline that the subsets $K$ and $K'$ are generally different.
	In particular, Lemma \ref{lem_spec_radius} and Proposition \ref{prop_spec_radius} do not determine the asymptotics of $\lambda_{\max}(T_k(f))$ and  $\lambda_{\min}(T_k(f))$ for general measures $\mu$.
	We formulate the corresponding question.
	\par 
	\textbf{Question}: do limits $\lim_{k \to \infty} \lambda_{\max}(T_k(f))$, $\lim_{k \to \infty} \lambda_{\min}(T_k(f))$ exist for measures $\mu$ that do not give full mass to pluripolar subsets? Do they coincide with $\max_{x \in K} f(x)$ and $\min_{x \in K} f(x)$ respectively?
	\par 
	Of course, the above question is closely related to the pointwise localization discussed in \cite{NevaiCond}.
	To explain this, for a given $x \in X$, we introduce the following measure
	\begin{equation}
		\mu_k^x := \frac{1}{B_k(x, x)} |B_k(x, y)|^{2}_{(h^L)^k} \cdot d\mu(y).
	\end{equation}
	Immediately from (\ref{eq_bk_id}), we see that $\mu_k^x$ is a probability measure. 
	Due to the same heuristics as described after (\ref{eq_bk_id}), one should expect that the measures $\mu_k^x$ converge weakly towards the Dirac mass at $x$, as $k \to \infty$, when $x \in {\textrm{supp}}(\mu)$.
	This naive expectation turns out to be false; for counterexamples, we refer the reader to Breuer-Last-Simon \cite[\S3-5]{NevaiCond}.
	Following \cite{NevaiCond}, measures that nevertheless satisfy this condition are said to satisfy the \textit{Nevai condition at $x$}.
	The following question generalizes \cite[Conjecture 1.4]{NevaiCond}, formulated for measures on $\mathbb{R} \subset \mathbb{P}^1$.
	\par 
	\textbf{Question}: for an arbitrary measure $\mu$ which does not give full mass to pluripolar subsets, does Nevai condition holds $\mu$-almost everywhere?
	\par 
	From Lemma \ref{lem_spec_radius}, it will then follow that if the answer to the above question is positive, then the answer to the question before it is positive as well.
	In fact, one can even show that an averaged version in the spirit of Theorem \ref{thm_off_diag} of the above question suffices for the needs of the first question.
	However, even the averaged version seems to be open in full generality.
	
	\subsection{Spectral equidistribution criteria for Toeplitz operators}\label{sect_toepl_eq}
	The main goal of this section is to generalize Theorem \ref{thm_distr} by showing that spectral equidistribution of Toeplitz operators is equivalent to the convergence of (diagonal) Bergman measure.
	We state this result in much greater generality when the measure $\mu$ itself is allowed to vary within a bounded family of measures.
	As a byproduct, we also generalize Theorems \ref{thm_off_diag} and \ref{cor_alg} to this setting.
	\par 
	We consider a sequence of Borel probability measures $\mu_k$, $k \in \nat$, on $X$, which is \textit{bounded} in the following sense: there is a continuous metric $h^L_0$ on $L$ so that for any $k \in \nat$,
	\begin{equation}\label{eq_asympt_bm_unif}
		{\textrm{Hilb}}_k(h^L, \mu_k)
		\geq
		{\textrm{Ban}}_k^{\infty}(h^L_0).
	\end{equation}
	By Proposition \ref{lem_asympt_bm}, for (\ref{eq_asympt_bm_unif}) to hold, it suffices that there is $C > 0$, so that $\mu_k \geq \exp(- C k) \cdot \mu$, for any $k \in \nat$, where $\mu$ is a measure which does not give full mass to pluripolar subsets.
	But we stress out that (\ref{eq_asympt_bm_unif}) is much more general, and it holds even for some measures $\mu_k$ with finite support.
	For instance, by \cite[Proposition 2.10]{BerBoucNys}, (\ref{eq_asympt_bm_unif}) holds when $\mu_k$ is a normalized sum of $\dim H^0(X, L^{\otimes k})$ Dirac masses associated with a Fekete configuration of points.
	\par 
	By an abuse of notation, we denote by $B_k(x, y) \in L^k_x \otimes (L^k_y)^*$ the Bergman kernel associated with ${\textrm{Hilb}}_k(h^L, \mu_k)$, defined as in (\ref{eq_bergm_kern}).
	It is well-defined by (\ref{eq_asympt_bm_unif}).
	We define the probability measures $\mu_{k}^{\rm{Berg}, \Delta}$ and $\mu_{k}^{\rm{Berg}}$ on $X$ and $X \times X$ as in (\ref{eq_mu_begmm}) and (\ref{eq_diag_bm}).
	For any $f \in \ccal^0(K, \real)$, we let $T_k(f)$ be the associated Toeplitz operator, defined for the Hermitian product ${\textrm{Hilb}}_k(h^L, \mu_k)$.
	We can now state the main result of this section.
	\begin{thm}\label{thm_equil_equidist}
	\begin{sloppypar}
		For any subsequence $k(i) \in \nat$, $i \in \nat$, of $\nat$, the following are equivalent:
		\begin{enumerate}[a)]
			\item
			The sequence $\mu_{k(i)}^{\rm{Berg}, \Delta}$ converges weakly, as $i \to \infty$, towards a measure $\nu$.
			\item
			The sequence $\mu_{k(i)}^{\rm{Berg}}$ converges weakly, as $i \to \infty$, towards the measure $\Delta_* \nu$, where $\Delta: X \to X \times X$ is the diagonal embedding.
			\item
			For any $f \in \ccal^0(X, \real)$, the spectral measures $\mu_{k(i)}^{Spec}(f) := \frac{1}{n_{k(i)}} \sum_{\lambda \in {\rm{Spec}} (T_{k(i)}(f))} \delta[\lambda]$ of $T_{k(i)}(f)$ converge weakly, as $i \to \infty$, towards the measure $f_* \nu$.
		\end{enumerate}
	\end{sloppypar}
	\end{thm}
	The proof of the above result will be based on the following generalization of Theorem \ref{thm_off_diag}.
	
	\par 
	\begin{thm}\label{thm_off_diag_gen}
		For any bounded (in the sense of (\ref{eq_asympt_bm_unif})) sequence of Borel probability measures $\mu_k$, the corresponding measures $\mu_k^{\rm{Berg}}$ do not asymptotically place mass away from the diagonal.
		In other words, for any compact subset $K \subset X \times X$, not intersecting the diagonal, we have (\ref{eq_thm_off_diag}).
	\end{thm}
	\begin{proof}
		We first note that the analogue of Theorem \ref{thm_pp_no_mass} continues to hold: for any closed pluripolar subset $E \subset X$ and any $\epsilon > 0$, there exists an open neighborhood $U$ of $E$, so that for any $k \in \nat$, we have (\ref{eq_pp_no_mass}).
		This is immediate as we have assumed that the bound (\ref{eq_asympt_bm_unif}) is verified, and this was the only input we had from the measure $\mu$ in the proof of Theorem \ref{thm_pp_no_mass}.
		Also, the analogue of Proposition \ref{prop_reduc} immediately holds, and the proof remains unchanged.
		Finally, if for two linearly independent sections $s_1, s_2 \in H^0(X, L^{\otimes k_0})$, there is a compact subset $K$ not intersecting the subvariety $Y := {\textrm{div}}(s_2)$, so that the supports of Borel probability measures $\mu_{k, 0}$ all lie within $K$, then the analogue of Theorem \ref{thm_qual_deloc} continues to hold: there is $C > 0$, so that for any $k \in \nat^*$, we have
		\begin{equation}\label{eq_thm_qual_deloc_gen}
			\iint \big|B_{k, 0}(x, y)\big|^2_{(h^L)^k} \cdot \Big| \frac{s_1(x)}{s_2(x)} - \frac{s_1(y)}{s_2(y)} \Big|^2 \cdot d \mu_{k, 0}(x) d \mu_{k, 0} (y)
			\leq 
			C \cdot (n_k - n_{k - k_0}),
		\end{equation}
		where $|B_{k, 0} (x, y)|^2_{(h^L)^k} d \mu_{k, 0}(x) d \mu_{k, 0} (y)$ is the Bergman measure, defined as in the beginning of Section \ref{sect_supp}.
		From this stage, the proof of Theorem \ref{thm_off_diag_gen} repeats the proof of Theorem \ref{thm_off_diag}.
	\end{proof}
	\par 
	We can now prove that Toeplitz operators associated with ${\textrm{Hilb}}_k(h^L, \mu_k)$ respect the asymptotic algebra property, generalizing Theorem \ref{cor_alg}.
	\begin{cor}\label{cor_alg_gen}
		For any bounded (in the sense of (\ref{eq_asympt_bm_unif})) sequence of Borel probability measures $\mu_k$, the space of Toeplitz operators verifies the asymptotic algebra property.
		In other words, for any $f, g \in \ccal^0(K, \real)$, $p \in [1, +\infty[$, we have (\ref{eq_toepl_comp_as}), where the $p$-Schatten norm is calculated with respect to the Hermitian product ${\textrm{Hilb}}_k(h^L, \mu_k)$.
	\end{cor}
	\begin{proof}
		The proof repeats the proof of Theorem \ref{cor_alg} verbatim.
		One only has to replace Theorem \ref{thm_off_diag} by the more general Theorem \ref{thm_off_diag_gen}.
	\end{proof}
	\begin{proof}[Proof of Theorem \ref{thm_equil_equidist}]
		An immediate verification shows that for any $f: X \to \real$, we have
		\begin{equation}
			\int f(x) \cdot d \mu_{k(i)}^{\rm{Berg}, \Delta}(x) 
			=
			\frac{1}{n_{k(i)}}
			\tr{T_{k(i)}(f)}
			=
			\int x \cdot d \mu_{k(i), f}^{Spec}(x).
		\end{equation}
		The implication $c) \Rightarrow a)$ is then immediate.
		The implication $b) \Rightarrow a)$ follows from (\ref{eq_push_forw}).
		The proof of the reverse implications $a) \Rightarrow c)$ and $a) \Rightarrow b)$ follows the same argument as in the proof of Theorem \ref{thm_distr}: with Theorem \ref{thm_off_diag} replaced by Theorem \ref{thm_off_diag_gen}.
	\end{proof}

	\subsection{Bergman measures on singular spaces}\label{sect_sing}
	We fix from now on a compact complex reduced irreducible analytic space $X$, $\dim_{\comp} X = n$, and a big line bundle $L$ over $X$.
	We fix a continuous Hermitian metric $h^L$ on $L$ and a positive \textit{volume form} $\mu := dV$ on $X$ (which means it is a pull-back of a (strictly) positive $(n, n)$-differential form in each chart).
	For the background on the singular complex spaces, we recommend \cite[\S II.9]{DemCompl}.
	We denote by $B_k(x, y)$ the associated Bergman kernel, defined as in (\ref{eq_bergm_kern}), and $\mu_k^{\rm{Berg}}$ as in (\ref{eq_mu_begmm}).
	\par 
	The asymptotics of $\frac{1}{k} \log B_k(x, x)$ has been investigated in a series of recent papers by Coman-Marinescu \cite{ComMar}, and by Coman-Ma-Marinescu \cite{ComanMaMarinGT}, \cite{ComanMaMarMoish}, see also Berman \cite[Theorem 1.5]{BermanEnvProj} and Bayraktar \cite[Proposition 2.9]{BayraktarEquidistr} for related results in this direction.  
	Our main focus in this section is on the non-logarithmic asymptotics, both on and off the diagonal.
	\par 
	In order to state our results, we define first the equilibrium measure associated with $(X, h^L)$.  
	To this end, let $\pi : \hat{X} \to X$ be a resolution of singularities.  
	According to \cite[Proposition 2.2.43]{LazarBookI}, the pullback line bundle $\pi^* L$ is big and
	\begin{equation}\label{eq_vol_resol}
    	\mathrm{vol}(\pi^* L) = \mathrm{vol}(L).
	\end{equation}
	We then define the equilibrium measure $\mu_{\mathrm{eq}}(X, h^L)$ on $X$ as
	\begin{equation}\label{eq_defn_eq_sing}
    	\mu_{\mathrm{eq}}(X, h^L) := \pi_* \mu_{\mathrm{eq}}(\hat{X}, \pi^* h^L).
	\end{equation}
	Note that, since any two resolutions of singularities can be dominated by a third one, the above definition is independent of the choice of resolution by \cite[Proposition 1.9]{BermanBouckBalls}.
	\par 
	For $k \in \nat$, we denote by $B_k(x, y)$ (resp. $\hat{B}_k(x, y)$) the Bergman kernel associated with $X$, $L$, $h^L$ and $\mu$ (resp. $\hat{X}$, $\pi^* L$, $\pi^* h^L$ and $\pi^* \mu$), defined as in (\ref{eq_bergm_kern}), and by $\mu_k^{\rm{Berg}}$ (resp. $\hat{\mu}_k^{\rm{Berg}}$) the measure on $X \times X$ (resp. $\hat{X} \times \hat{X}$), defined as in (\ref{eq_mu_begmm}).
	We now state the main result of this section.
	\par 
	\begin{thm}\label{thm_sing}
	\begin{sloppypar}
		The sequence of measures $\mu_k^{\rm{Berg}}$ on $X$ converges weakly, as $k \to \infty$, to $\Delta_* \mu_{\mathrm{eq}}(X, h^L)$, where $\Delta: X \to X \times X$ is the diagonal embedding.
		Moreover, the analogues of Theorems \ref{cor_alg} and \ref{thm_distr} hold in this setting.
	\end{sloppypar}
	\end{thm}
	Let us begin by establishing some preliminary results.
	For brevity, we denote 
	\begin{equation}
		\mu_k^B := \frac{1}{n_k} B_k(x, x) \cdot d \mu(x), \qquad \hat{\mu}_k^B := \frac{1}{\hat{n}_k} \hat{B}_k(x, x) \cdot d \hat{\mu}(x),
	\end{equation}
	where $\hat{n}_k := \dim H^0(\hat{X}, \pi^* L^{\otimes k})$ and $n_k := \dim H^0(X, L^{\otimes k})$.
	\begin{lem}\label{lem_comp_res_meas}
		As $k \to \infty$, the sequence of measures $\pi_* \hat{\mu}_k^B - \mu_k^B$ (resp. $(\pi, \pi)_* \hat{\mu}_k^{\rm{Berg}} - \mu_k^{\rm{Berg}}$) on $X$ (resp. on $X \times X$) converges weakly to zero.
	\end{lem}
	\begin{proof}
		According to (\ref{eq_vol_resol}), $n_k \sim \hat{n}_k$, as $k \to \infty$, and so it suffices to establish that the measures
		\begin{equation}\label{eq_meas_aux}
		\begin{aligned}
			&
			\frac{1}{n_k} \Big( \hat{B}_k(x, x) - B_k(\pi(x), \pi(x)) \Big) \cdot d \hat{\mu}(x),
			\\
			&
			\frac{1}{n_k} \Big( |\hat{B}_k(x, y)|^2_{(\pi^* h^L)^k} - |B_k(\pi(x), \pi(y))|^2_{(h^L)^k} \Big) \cdot d \hat{\mu}(x) d \hat{\mu}(y),
		\end{aligned}
		\end{equation}
		converge weakly to zero, as $k \to \infty$.
		To begin with, we concentrate on the first convergence.
		From the definitions, we have 
		\begin{equation}
			\int_{\hat{X}}  \Big( \hat{B}_k(x, x) - B_k(\pi(x), \pi(x)) \Big) \cdot \pi^* d \mu(x)
			=
			\dim H^0(\hat{X}, \pi^* L^{\otimes k})
			-
			\dim H^0(X, L^{\otimes k}).
		\end{equation}
		Since by (\ref{eq_vol_resol}), the right-hand side of the above equality is asymptotically negligible compared to $n_k$, we deduce that the mass of the sequence of measures (\ref{eq_meas_aux}) tends to zero as $k \to \infty$.  
		However, the embedding of $H^0(X, L^{\otimes k})$ into $H^0(\hat{X}, \pi^* L^{\otimes k})$ is isometric with respect to the associated $L^2$-products, so the difference $\hat{B}_k(x, x) - B_k(\pi(x), \pi(x))$ is positive as it can be written as $\sum |s_i(x)|^2_{(h^L)^2}$ where $s_i$ form a basis of the orthogonal complement to the image of $H^0(X, L^{\otimes k})$ in $H^0(\hat{X}, \pi^* L^{\otimes k})$. 
		Hence the measures in (\ref{eq_meas_aux}) are positive, and as their mass tends to zero, the measures themselves tend weakly to zero as $k \to \infty$.
		\par 
		The proof of the second convergence from (\ref{eq_meas_aux}) is similar, and we only highlight the main steps. 
		We write 
		\begin{equation}
			\hat{B}_k(x, y) 
			=
			B_k(\pi(x), \pi(y))
			+
			R_k(x, y),
		\end{equation}
		where $R_k$ is the Schwartz kernel of the projection onto the orthogonal complement of $H^0(X, L^{\otimes k})$ in $H^0(\hat{X}, \pi^* L^{\otimes k})$.
		Then by (\ref{eq_vol_resol}), for any $\epsilon > 0$, there is $k_0 \in \nat$ so that for any $k \geq k_0$, we have 
		\begin{equation}
			\int |R_k(x, y)|^2_{(\pi^* h^L)^k} \cdot d \hat{\mu}(x) d \hat{\mu}(y)
			\leq
			\epsilon \cdot n_k.
		\end{equation}
		By this and Minkowski inequality, we deduce 
		\begin{equation}
			\bigg|
			\sqrt{ \int |\hat{B}_k(x, y)|^2_{(\pi^* h^L)^k} \cdot d \hat{\mu}(x) d \hat{\mu}(y)} 
			- 
			\sqrt{ \int |B_k(x, y)|^2_{(h^L)^k} \cdot d \mu(x) d \mu(y)}
			\bigg|
			\leq
			\sqrt{\epsilon n_k}.
		\end{equation}
		From this, it follows immediately that
		\begin{equation}
			\Big|
			\int |\hat{B}_k(x, y)|^2_{(\pi^* h^L)^k} \cdot d \hat{\mu}(x) d \hat{\mu}(y)
			- 
			\int |B_k(x, y)|^2_{(h^L)^k} \cdot d \mu(x) d \mu(y)
			\Big|
			\leq
			\sqrt{2 \epsilon} \hat{n}_k,
		\end{equation}
		implying the second inequality from (\ref{eq_meas_aux}).
	\end{proof}
	\begin{lem}\label{lem_bm}
		The measure $\pi^* \mu$ on $\hat{X}$ is Bernstein-Markov with respect to $\pi^* h^L$.
	\end{lem}
	In order to establish Lemma \ref{lem_bm}, we need to recall some notions from pluripotential theory.
	We thus fix a complex manifold $Z$ with a big line bundle $F$ endowed with a continuous metric $h^F$.
	We fix a Borel measure $\nu$ with support $Z$ (all the definitions and results below work for measures supported on arbitrary compact subsets, but since we shall only work with fully supported measures, we simplify our presentation accordingly).
	Following \cite{SiciakDetermMeas} and \cite{BerBoucNys}, recall that $\nu$ is called \textit{Bernstein-Markov with respect to psh weights on $(Z, h^F)$} if for any $\epsilon > 0$, there is $C > 0$, such that for any $\psi: X \to [- \infty, +\infty[$, so that $h^F \cdot \exp(- \psi)$ has a psh potential, for any $p \geq 1$, we have
	\begin{equation}\label{bm_psh_weights}
		\sup_Z \big( \exp(p \cdot \psi) \big)
		\leq
		C
		\cdot
		\exp(\epsilon p)
		\cdot
		\int_Z \exp(p \cdot \psi) d \nu.
	\end{equation}
	\par 
	It is immediate that $\nu$ is Bernstein-Markov for $(Z, h^F)$ if it is Bernstein-Markov with respect to psh weights.
	Indeed, by substituting $\psi := \frac{1}{k} \log |s|_{(h^F)^k}$ and $p = 2k$ into (\ref{bm_psh_weights}), one obtains the required bound for (\ref{eq_bm_clas}).
	\par 
	Following \cite{BermanBouckBalls}, we say that $\nu$ is \textit{determining} for $(Z, h^F)$ if for each measurable subset $E \subset Z$, $\nu(E) = 0$, we have $h^F_{Z} = h^F_{Z \setminus E}$, where we used the notation (\ref{defn_env}).
	This definition differs slightly from \cite{SiciakDetermMeas}. 
	For a comparison, we refer to \cite[Proposition 3.6]{FinEigToepl}.
	\par 
	Let us recall the following result from \cite[Theorem 1.14]{BerBoucNys} (see also \cite[Theorems 5.1 and 5.2]{SiciakDetermMeas}).
	\begin{thm}\label{thm_bbwn_determining}
		The following statements are equivalent: 
		\begin{enumerate}[a)]
			\item The measure $\nu$ is determining for $(Z, h^F)$.
			\item The measure $\nu$ is Bernstein-Markov with respect to psh weights for $(Z, h^F)$.
		\end{enumerate}
	\end{thm}
	\begin{rem}\label{rem_equiv_class}
		Note that the definition of a determining measure depends only on the absolute continuity class of the measure. By Theorem \ref{thm_bbwn_determining}, the same holds true for the Bernstein-Markov property with respect to psh weights.
	\end{rem}
	\begin{proof}[Proof of Lemma \ref{lem_bm}]
		We will in fact establish a stronger statement by showing that the measure $\pi^* \mu$ on $\hat{X}$ is Bernstein-Markov with respect to psh weights on $(\hat{X}, \pi^* h^L)$.  
		First, observe that by the mean-value inequality, any positive volume form on $\hat{X}$ is Bernstein-Markov with respect to psh weights for any line bundle endowed with a continuous metric.  
		For a proof, the reader may consult \cite[Lemma 2.2]{BermanBouckBalls}, which treats the classical Bernstein-Markov property; the same argument applies verbatim in the plurisubharmonic setting.  
		Note also that the measure $\pi^* \mu$ lies in the same absolute continuity class as any positive volume form on $\hat{X}$.  
		We conclude by Remark \ref{rem_equiv_class}.
	\end{proof}
	\begin{proof}[Proof of Theorem \ref{thm_sing}]
		Observe that Theorem \ref{thm_distr} is a formal consequence of Theorem \ref{cor_alg}, which itself follows from Theorem \ref{thm_off_diag}. 
		Thus, it is enough to prove the analogue of the latter result, on which we concentrate from now on.
		\par 
		By Theorem \ref{thm_off_diag_bm} and Lemma \ref{lem_bm}, as $k \to \infty$, the sequence of measures $\hat{\mu}_k^{\rm{Berg}}$ converges to $\hat{\Delta}_* \mu_{\mathrm{eq}}(\hat{X}, \pi^* h^L)$.
		We conclude by this, Lemma \ref{lem_comp_res_meas}, (\ref{eq_defn_eq_sing}) and the fact that pushforwards under continuous maps preserve weak convergence.
	\end{proof}

\bibliography{bibliography}

		\bibliographystyle{abbrv}

\Addresses

\end{document}